\documentclass[12pt, a4paper]{amsart}
\usepackage{amsfonts, amssymb, amsmath}
\usepackage{amsthm}
\usepackage[margin=2.5cm]{geometry}
\usepackage{enumerate}
\usepackage[all]{xy}
\usepackage{color, xcolor}
\usepackage{showlabels}
\usepackage{longtable}

\definecolor{gn}{RGB}{40, 140, 80}
\definecolor{gnn}{RGB}{40, 200, 40}

\newcommand{\Q}{{\mathbb Q}}
\newcommand{\R}{{\mathbb R}}
\newcommand{\Z}{{\mathbb Z}}

\newcommand{\A}{{\mathbb A}}
\newcommand{\bA}{{\mathbb A}}
\newcommand{\bC}{{\mathbb C}}
\newcommand{\bx}{{\mathbf x}}
\newcommand{\by}{{\mathbf y}}

\newcommand{\bw}{{\mathbf w}}
\newcommand{\bu}{{\mathbf u}}
\newcommand{\bv}{{\mathbf v}}

\newcommand{\cE}{\mathcal{E}}

\newcommand{\fa}{{\mathfrak{a}}}
\newcommand{\fb}{{\mathfrak{b}}}

\newcommand{\Cc}{{\mathbb C}}
\newcommand{\Oo}{\mathcal{O}}
\newcommand{\cen}{\mbox{Center}}
\newcommand{\se}[2]{\left\lbrace #1 \mbox{ }\vline\mbox{ } #2 \right\rbrace}

\newcommand{\tl}[1]{\tilde{#1}}

\newcommand{\st}{^{\ast}}
\newcommand{\ts}{_{\ast}}

\newcommand{\ang}[1]{\langle #1\rangle}
\newcommand{\du}{^{\vee}}
\newcommand{\spc}{\mbox{Spec }}

\newtheorem{thm}{Theorem}[section]
\newtheorem{pro}[thm]{Proposition}
\newtheorem{cor}[thm]{Corollary}
\newtheorem{lem}[thm]{Lemma}
\theoremstyle{definition}
\newtheorem{rk}[thm]{Remark}

\newtheorem{cov}[thm]{Convention}
\newtheorem{defn}[thm]{Definition}

\newtheorem{prop}[thm]{Proposition}

\begin{document}

\title{On the  divisorial contractions to curves of threefolds}
\author{Hsin-Ku Chen, Jheng-Jie Chen \and Jungkai A. Chen}
\date{}
\begin{abstract}
We prove that each divisorial contraction to a curve between terminal threefolds is a weighted blow-up under a suitable  embedding. Moreover, we give a classification of the weighted blow-ups assuming that the curve is smooth.
\end{abstract}
\maketitle
\section{Introduction}
Minimal model program has been a very important tool in birational geometry. The purpose of minimal model program, roughly speaking, is to find a good birational model inside a birational equivalent class. To achieve this goal, one needs to consider various birational maps including divisorial contractions, flips and flops. The minimal model conjecture predicts that one can reach either a minimal model or a Mori fiber space after finitely many steps of divisorial contractions and flips. Moreover, it is proved by Kawamata that minimal models are connected by flops (cf. \cite{Kawamata08}).  

In dimension three, minimal model program is considered to be well-established: minimal model exists and also abundance conjecture holds. However, their explicit descriptions are still under development and not completely well-understood. 

Let us recall some results along this direction. For divisorial contractions to points, the complete classification follows from a series of work mainly due to Kawamta, Hayakawa, Kawakita, Yamamoto. For details, please refer to  \cite{Kaw96, Hayakawa99, Hayakawa00, Kawakita02, Kawakita05,  Yamamoto} and the references therein. As a consequence, any of the divisorial extraction over a terminal singularity can be realized as weighted blowup under suitable embedding into an affine space (or its cyclic quotient) of dimension $\le 5$. 

It is thus natural and important to understand threefold divisorial contractions to curves.
In Mori's seminal work \cite{Mori82}, he proved that a divisorial contraction to curve from a smooth threefold $Y$ is the usual blowdown to a smooth curve inside a smooth threefold $X$. Later, Cutkosky \cite{Cut88} generalized the result by showing that a divisorial contraction from a threefold with at worst Gorenstein singularity is the usual blowdown to a lci curve inside a smooth threefold $X$. 

In general, let us consider a divisorial contraction $\phi: Y \to X$ contracting an exceptional divisor $E$ to a curve $C$. Away from possibly finitely many critical points $\Sigma \subset C$, it is well-known that $Y - \phi^{-1}(\Sigma) \to X -\Sigma$ is nothing but the blowup along $C -\Sigma$. Therefore, it suffices to restrict ourselves to  an analytic neighborhood of a point $P \in \Sigma$ in $X$ and then study the divisorial contraction over this neighborhood.

The most systematical study of divisorial contraction to curve was due to Tziolas.  In his series of works (cf. \cite{Tz03, Tz05a, Tz05b, Tz10}), he worked out the situation that $P \in C \subset S \subset X$
that $C$ is a smooth curve and $S$ is a general hypersurface passing through $P$ with at worst DuVal singularities. Tom Ducat considered different situation of a singular curve $C$ contained in a smooth 3-fold $X$. Assuming the general elephant conjecture, he  proved the non-existence  if the general elephant $S$ is of type $D_{2k}$,$E_7$ or $E_8$ and he classified the cases that  $S$ is of type $A_1$,$A_2$ or $E_6$ (cf. \cite{Duc16}). 

The purpose of this article is to study threefold divisorial contractions to curve from a different approach. Unlike the previous work of Tziolas and Ducat, in which general elephants play the crucial role. 
We consider an algebraic approach by constructing appropriate weighted blouwps. 

Suppose that $P \in X$ is a terminal singularity of Cartier index $1$ (resp. $>1$). By the classification of terminal singularities, $X$ (resp. canonical cover of $X$) has embedded dimension $ \le 4$. Fix an embedding $\jmath: X \hookrightarrow \A^4/\mu_r=:W/\mu_r$, where $r$ is the Cartier index. We may assume that $X$ defined by some regular function $f$. Assign appropriate weights of non-negative integers $\bw:=(a_1, a_2, a_3, a_4)$ to each coordinate, we can consider the weighted blowup $\pi : \tl{W} \to W$. In the best scenario that the proper transform of $X$ in $\tl{W}/\mu_r$ is isomorphic to $Y$, then we see that the given divisorial contraction $\phi: Y \to X$ is realized by a weighted blowup. 

However, this is not always the case. The best we can hope for is the following theorem.


\begin{thm} Given a threefold divisorial contraction to a smooth curve $\phi: Y \to X$ near a terminal singularity $P \in X$ of Cartier index $r \ge 1$. There exists an embedding $\jmath': P \in X \hookrightarrow o \in W'/\mu_r \cong \A^N/\mu_r$ and  weights $\bw'=(a_1, \ldots, a_N)$ so that $Y$ is isomorphic to the weighted blowup of $X$ with weights $\bw'$. 

More precisely, consider the weighted blowup $\pi_{\bw'}: \tl{W}' \to W'$. The $\mu_r$-action on $W$ extends to a $\mu_r$-action on $W'$ so that  $Y$ is isomorphic to the proper transform of $X$ in $\tl{W}'/\mu_r$.  
\end{thm}

We briefly explain the idea of the proof in the case that $r=1$. The case that $r>1$ follows from the same argument with some extra care of making the whole process compatible with $\mu_r$ action.  
We start by considering the associated valuation $v_\bw$ on $\Oo_W$ and then restrict it to $\mathcal{O}_X$ via the embedding map $\jmath$. We compare this {\it restriction of the valuation} associated to $(\jmath,\bw)$ with the valuation $\mu_E$ induced by the exceptional divisor $E \subset Y$. If they do not agree, then we modify the embedding and weights by introducing a {\it tilting}, in which we introduce one more new variable and one more weight. By certain finite generation property, one sees that only finitely many steps of tiltings are allowed. In other words, the above theorem follows.

We developed foundations of weighted blowups and comparison of valuations in Section 2 and then we apply these to prove our first main theorem in Section 3.

The second part of this article aims to classify threefold divisorial contractions to curves. To this end, we develop an algorithm of tilting. Roughly speaking, we start with an embedding into $\A^4$ defined by $f$ and choose weights of coordinate $x_i$ by $\mu_E(x_i)$. The proper transform and the exceptional set of the weighted blowup can be computed easily for hypersurface. However, it could be a lot more subtle if we embed the variety into a higher dimensional affine space. Our algorithm suggest to consider nice generating set of the ideal of $X$ in $\A^N$ which we call a {\it $v$-basis}. Then the exceptional set is then defined by the homogeneous part (with respect to the given weight) of $v$-basis. One can verify that exceptional set is irreducible if and only if the restricted valuation coincide with the given valuation, hence we can realized the divisorial contraction as a weighted blowup as desired in this situation and the algorithm stopped.

Turning into more details, we describe the idea of classification under the assumption that $C \subset  \A^4=:W$ is a smooth curve. After changing of coordinates, we may assume that $C$ is the $x_1$-axis. The first attempt is consider weighted blowup with weights $\bw$ assigned by $a_i=\mu_E(x_i)$, for $i=1,...4$. Let $f$ be the defining equation of $X$ in $W$ near $o \in X$. 
Suppose that $f_{v}$, the homogeneous part of $f$ of lowest weight, is irreducible. The weighted blowup of $X$, that is the proper transform of $X$ in the weighted blowup of $W$, has irreducible exceptional divisor. This gives an divisorial contraction as long as it has at worst terminal singularities.  If $f_{v}$ is reducible, then we introduce an tilting. However, since $f$ defines a terminal singularities, one has $\textrm{mult}_o f=2$ and some other constraints. Roughly speaking, there are several conditions on terms of degree $2$ and $3$. The existence of such terms allows us to see that the tilting algorithm which we will introduced in Section 4 will terminated after at most one tilting for all possible types of $f$. The classification will be done in Section 6.  The classification result can summarized as follows.
\begin{thm}\label{thm2}
Let $\phi: Y \to X \ni o$ be a divisorial contraction to a curve $C$ over a germ $o \in X$.  Suppose that $o \in X$ is Gorenstein terminal and $C$ is smooth (resp. $o \in X$ is non-Gorenstein terminal and the preimage of 
 $C$ in its canonical cover is irreducible), then $\phi$ is isomorphic to one of those in the Table at the end, which is a weighted blowup in an ambient space of dimension $\le 5$. 
\end{thm}

 In this article, we only work out the case that $C$ is a smooth curve in $X$, as well as the case that the preimage of curve $C$ in the canonical cover is smooth (in the case $o \in X$ is a non-Gorenstein terminal singularity).
 
 On the other hand, there are some extra technical issues if the contracted curve is singular. So that our current tilting algorithm is not able to work verbatim in some cases. We will leave the classification of divisorial contraction to singular curve to subsequent works.

 As an application, we prove the general elephant conjecture for divisorial contractions to smooth curves.

\begin{thm}\label{gethm}
Let $Y \to X$ be a threefold divisorial contraction to a smooth curve $C$ satisfying the hypothesis of Theorem \ref{thm2}. Then there exists a surface $S$ containing $C$ with at worst Du Val singularities. This implies in particular that a general surface containing $C$ has at worst Du Val singularities. 
\end{thm}

The general elephant conjecture for divisorial contractions to curves between terminal threefolds had been proved by Koll\'{a}r and Mori \cite[Theorem 2.2]{km92} in the case that every fiber is irreducible. This conjecture is also known to be true by experts if the underlying space is analytically $\Q$-factorial. Nevertheless, once we have the classification, there is a straight forward proof for the existence of general elephants, without any further assumption.

\section*{Acknowledgement}
We are grateful to Paolo Cascini, J\'anos Koll\'ar, Ching-Jui Lai, Yongnam Lee, Hsueh-Yung Lin, Shigefumi Mori, Nikolas Tziolas, Chenyang Xu       for helpful discussion. The first author is supported by KIAS individual Grant MG088901. The second and third authors are partially support by National Center for Theoretical Sciences and  National Science and Technology Council of Taiwan.  

 \section{Preliminaries}

In this section, we developed some foundation of our construction. Most of the materials are either easy extension or variation of known formulations. We collect and present them here so that it might be helpful to the readers. 

We work over complex number $\mathbb{C}$.

\subsection{Classification of threefold terminal singularities}
We need the classification of threefold terminal singularities especially the description of $cDV$ singularities. Thus we briefly recall the classification mainly due to Reid and Mori. 

\begin{thm}
    A threefold terminal singularity of Cartier index $1$ is an isolated cDV singularity. More precisely, analytical locally it is an isolated hypersurface singularity defined by $f(x,y,z,u)=g(x,y,z)+uh(x,y,z,u)$, where $g(x,y,z)$ defines a surface Du Val singularity. 
\end{thm}

\begin{thm}
   Let $X$ be a germ of threefold terminal singularity of index $r>1$. Then one of the following holds:
   \begin{enumerate}
       \item  (cA/r) $X \cong (xy+h(z,u)=0) \subset \Cc^4/\frac{1}{r}(a, r-a, 1, 0)$ for some $(a,r)=1$.
       \item (CAx/2) $X \cong (x^2+y^2+h(z,u)=0) \subset \Cc^4/\frac{1}{2}(1, 0, 1, 1)$.
       \item (cAx/4) $X \cong (x^2+y^2+h(z,u)=0) \subset \Cc^4/\frac{1}{4}(1, 3, 1, 2)$.
       \item (cD/2) $X \cong (f=x^2+g(y,z,u)=0) \subset \Cc^4/\frac{1}{2}(1, 1, 1, 0)$ so that $f$ is an isolated $cD$ singularity.
       \item (cD/3) $X \cong (f=x^2+y^3+g(z,u)y+h(z,u)=0) \subset \Cc^4/\frac{1}{2}(0, 1, 2, 2)$ so that $f$ is an isolated $cD$ singularity.
       \item (cE/2) $X \cong (f=x^2+y^3+g(z,u)y+h(z,u)=0) \subset \Cc^4/\frac{1}{2}(1, 0, 1, 1)$ so that $f$ is an isolated $cE$ singularity.
   \end{enumerate}
\end{thm}

It is convenient to introduce {\it axial weight} of a singularity. For example, in the case of $cA/r$, the axial weight is defined to be $\min\{k|u^k \in h\}$.

\subsection{Weighted blowups}
Let $G=\langle \tau \rangle $ be a cyclic group of order $r$ generated by $\tau$. Fix an coordinates system $\bx=(x_1,...,x_n)$ of $\A^n$.
For any $\Z$-valued $n$-tuple $(b_1,...,b_n)$
one can define a $G$-action on $\A^n_{\bx}$ by $\tau(x_i)=\zeta_r^{b_i}x_i$, where $\zeta_r$ is a primitive $r$-th root of unity.
We will denote the quotient space $\A^n_\bx/G$ by $\A^n_{\bx}/\frac{1}{r}(b_1,...,b_n)$ or simply $\A^n/\frac{1}{r}(b_1,...,b_n)$.\par

We introduce the weighted blowup by using the language of toric varieties. Let $\{e_1,...,e_n\}$ be the standard basis of $\R^n$ and $\sigma$  (resp. $N_0$) be the cone (resp. lattice) generated by $\{e_1,...,e_n\}$. 
We set $\bv=\frac{1}{r}(b_1,...,b_n)$ and $N:=N_0+\mathbb{Z} \bv$. Then \[W:=\A^n/\frac{1}{r}(b_1,...,b_n) \cong \spc\Cc [N\du\cap \sigma\du]. \]

By a weight, we mean a vector $\bw=\frac{1}{r}(a_1,...,a_n)$  in $N$ such that $a_i \ge 0$ for all $i$. One can consider the barycentric subdivision of the cone $\sigma$ along the ray  $\mathbb{R}_{\ge 0} \bw$. For each $a_i \ne 0$, one has cones 
\[\sigma_i=\ang{e_1,...,e_{i-1},\bw,e_{i+1},...,e_n}.\]

Let $U_i= \spc\Cc [N\du \cap \sigma_i\du]$ be the toric variety defined by the cone $\sigma_i$. The inclusion of $\sigma_i \hookrightarrow \sigma$ yields a morphism $U_i \to W$. Patching them altogether, then the natural map \[\pi_{\bw}: \widetilde{W}=\bigcup_{a_i \ne 0} U_i \rightarrow W \]  is called the weighted blowup of $W$ with weight $\bw$.

We describe the singularity type of $U_i$. Consider the sublattice $N_\bw$ and $N'$ defined below and the inclusions \[N_\bw:=\ang{e_1,...,e_{i-1},\bw,e_{i+1},...,e_n}_\Z \subset N':= N_\bw + \Z e_i = N_0+\Z \bw \subset  N.\] 
It is clear that $[N:N']$ divides $r$ and $[N':N_\bw]=a_i$. We set $r_i:=[N:N_\bw]$. It is clear that $N/N_\bw$ is generated by $\bv$ and $e_i$. 

\begin{lem}\label{quot} $U_i$ is a cyclic quotient space if and only if there is a vector $\bu$ satisfying $N=N_\bw+\Z \bu$. In this case, one has 
 $r_i \bu =b'_i \bw +  \sum_{j\ne i} {b'_j} e_j$ for some $b'_j$s, and hence \[U_i\cong \A^n/\frac{1}{r_i}(b'_1,..., b'_n).\]
\end{lem}

It is clear that $N/N_\bw$ is generated by $\bv$ and $e_i$. Therefore, it is convenient to consider the following criterion. 
\begin{lem}\label{nonquot} If   $\bv \ne e_i$ (resp. $\bv \ne e_i^{j}$ for $j=1,2$) and both has order $2$ (resp. $3$) in $N/N_\bw$, then $U_i$ is not a cyclic quotient space. 
\end{lem}

\begin{proof}
    In a cyclic group of order $2n$, there is exactly one element of order $2$. Similarly, in a cyclic group of order $3n$, there are exactly two elements of order $3$.  
\end{proof}

\begin{defn}
    Keep the notation as above. The group $G$ is said to be {\it admissible} if $N/N_\bw$ is cyclic for all $i$. 
\end{defn}

The weighted blowups can be computed explicitly. The reader may check refer to \cite{Hayakawa99} for some more detailed description that meet our purpose here. For example, we have the following useful formula (cf. \cite[p. 521]{Hayakawa99})

\begin{lem}\label{xij}
	Let $x_1$,\ldots, $x_n$ be the local coordinates of $W$ and $\bar{x}_1, \ldots, \bar{x}_{n}$ be the local coordinates of $U_i$.
	The map $U_i\rightarrow W$ is given by $x_j=\bar{x}_j \bar{x}_i^{\frac{a_j}{r}}$ and $x_i=\bar{x}_i^{\frac{a_i}{r}}$.
\end{lem}

\subsection{Proper transform}
\begin{defn} Let $W$ be an affine space with coordinates $\bx=(x_1,\ldots, x_n)$ and let $\hat{\Oo}_W$ be the completion of $\Oo_W$. We fix weight $\bw=(a_1,\ldots,a_n)$ of non-negative integers. Then there is a valuation on $\Oo_W$ (resp. $\hat{\Oo}_W$), denoted $v_\bw$, defined by 
\[v_\bw(f)=\min\{\sum_{j=1}^n a_j i_j | c_I \ne 0 \}, \]
for any $f=\sum c_I x^I$.
We call such valuation  an {\it orthogonal valuation} (associated to $\bw$). 
\end{defn}

Indeed, this valuation is the same as the valuation induced by the exceptional divisor of the weighted blowup $\pi_\bw: \widetilde{W} \to W$.  

For any given $f \in \Oo_W$ (resp. $\hat{\Oo}_W$),
we may decompose 
\[ f= \sum_{n=0}^N f_{v_\bw=n} \quad( \text{resp. } f=\sum_{n=0}^\infty f_{v_\bw=n}), \] where $f_{v_\bw=n}$ denotes the part with weight $n$. 

In particular, we may frequently decompose
\[ f= f_{v_\bw} + f_{v_\bw>}, \]
where $f_{v_\bw}$ (resp. $f_{v_\bw>}$) is the part with weight $v_\bw(f)$ (resp. $>v_\bw(f)$).  

\begin{cov}
Throughout this section, we may write $v$ instead of $v_\bw$ and similarly 
$f_v$ instead of $f_{v_\bw}$ if no confusion is likely.  
\end{cov}

Let $\pi_\bw: \tl{W} \to W$ be a weighted blowup of weight $\bw$. Let $U_i$ be $i$-th affine piece of $\tl{W}$. For a  subvariety $V \subset W$ with given defining equations,  we aim to determine its proper transform (in $U_i$) explicitly.   
Without loss of generality, we assume that $a_1>0$ and consider $U_{1}$-chart. In this chart, the map
$\pi|_{U_{1}} : U_{1} \to W$ is given by $x_1 \mapsto \bar{x}_1^{a_1}$ and $x_i \mapsto \bar{x}_i \bar{x}_1^{a_i}$ for all $i \ne 1$.

\begin{defn}For any $f \in \Oo_W$, we define $\tilde{f} \in \Oo_{U_{1}}$ satisfying $f \circ \pi = \tilde{f} \cdot \bar{x}_1^m$ so that $\bar{x}_1 \nmid \tilde{f}$.
Explicitly,
\[\tilde{f}(\bar{x}_1,\ldots,\bar{x}_n):=f(\bar{x}_1^{a_1}, \bar{x}_2 \bar{x}_1^{a_2}, \ldots, \bar{x}_n \bar{x}_1^{a_n}) \bar{x}_1^{-m},\] where $m=v_\bw(f)$.
We call $\tilde{f}$ a {\it lifting of $f$} (on $U_{1}$-chart). Sometimes, we may also denote it as $\pi_*^{-1}f$  (on $U_{1}$-chart).
\end{defn}

Notice  that the lifting of $f \cdot g$ satisfying $\widetilde{f \cdot g}= \tilde{f}  \cdot \tilde{g}$. However, $\widetilde{f+g}= \tilde{f} + \tilde{g}$ holds only when $v_\bw(f)=v_\bw(g)$.

Notice also  that
\begin{equation}
    \tilde{f}(\bar{x}_1,\ldots,\bar{x}_n) \equiv \tilde{f}_v(\bar{x}_1,\ldots,\bar{x}_n) \quad (\textrm{mod } \bar{x}_1).
\end{equation}

\begin{cov}\label{cov} 
To simplify notations, we might abuse the notation by denoting the coordinates on $U_i$ as $x_j$ (instead of $\bar{x}_j$).

In particular, suppose that $f=f_v$ is homogeneous with respect to $\bw$ and not divisible by $x_i$. Then $\tl{f}=f$, in the sense that if $f=\sum c_I x^I$, then $\tl{f}=\sum c_I \bar{x}^I$. 

In most of the situations, there is no essential differences between $\Oo_W$, $\hat{\Oo}_W$ or analytic functions. Under such situation, we may use $R_W$ to denote any one of the rings. 
\end{cov}

\begin{lem} \label{eqn}
Keep the notation as above.  For any irreducible subvariety $V \subset W$ defined by $I$. Then its proper transform on $U_{i}$-chart is defined by $\tilde{I}:=\langle  \tilde{f}  \rangle_{f \in I }$.
\end{lem}

\begin{proof} Let $\cE$ be the exceptional divisor of $\pi_\bw$. Then $ \cE \cap U_i \subset U_{i}$ is defined by $\bar{x}_i=0$. Let $\widetilde{V}$ be the proper transform of $V$ in the $U_{i}$-chart and $J$ be the defining ideal of $\widetilde{V}$ in $U_i$.   

For any point $Q \in \widetilde{V}- \widetilde{V} \cap \cE $, its image $P:=\pi(Q) \in V -V \cap \pi_\bw(\cE)$. One has $f(P)=0$ for all $f \in I$. Hence
\[ 0= f(P)= f(\pi (Q))= \tilde{f}(Q) \cdot \bar{x}_i^m(Q).\] Therefore, $\tilde{f}(Q)=0$ for all $\tilde{f} \in \tilde{I}$ and hence $\tilde{I} \subset J$.

Given $g \in J$, we set a divisor $D: =\text{div}(g)$ on $W'$ and define $\pi_{*}(g) \in \Oo_W$ satisfying $\text{div}(\pi_{*}(g))=\pi_{*} D$, which vanishes on $V$. As $\pi^* \pi_{*}(D)=D+m \cE$ for some $m \ge 0$, one has 
\[\pi^* \pi_{*}g= (\pi_{*} g) \circ \pi = u \cdot g \cdot \bar{x}_i^m, \] for some unit $u$. Hence $u g \in \tilde{I}$ and so is $g$.
 \end{proof}


In order to determine the proper transform more effectively, we introduce the following notion.

\begin{defn} Fix an ideal $I \lhd R_W$, we can consider
\[ I_{v_\bw}:=\{ f_{v_\bw}| f=f_{v_\bw}+f_{v_\bw>} \in I\}.\]
A generating subset $\{ f_1,...,f_k\}$ of $I \lhd R$ is said to be a {\it $v_\bw$-basis} or simply a {\it v-basis} (with respect to $\bw$)  if $\{f_{1,v},...,f_{k,v}\} $ generate the ideal $\langle I_{v_\bw} \rangle$.
\end{defn}

The existence of finite  v-basis follows from the Noetherian condition of $R_W$. \footnote{The convergent formal power series ring is also Noether, see e.g. \cite[p. 72]{GunningRossi}. }

\begin{thm} Let $\{ f_1,...,f_k\}$ be a v-basis of $I_X \lhd \Oo_W$. Then
\begin{enumerate}
\item The proper transform $\widetilde{X}$ of $X$ in $\widetilde{W}$ is defined by  $\{ \tl{f}_1,...,\tl{f}_k\} $ in the completed local ring (of the $U_i$-chart).
\item The exceptional set of $\widetilde{X} \to X$ is defined by $\{f_{1,v},...,f_{k,v} \}$, as a subvariety of $\mathbb{P}(a_1,\ldots, a_n)$ and under the Convention \ref{cov}.
\end{enumerate}

\end{thm}

\begin{proof} Let $R \colon =\Oo_{U_i}$ and $\hat{R}$ be its completion. By Lemma \ref{eqn}, we know that the defining ideal $\tl{I}$ of $\tl{X}$ (on a chart) is generated by $\{\tl{f}| f \in I \}$ in $R$ and in $\hat{R}$ as well. We aim to show that for any $u \in I$, one has $\tl{u} \in \langle \tl{f}_1,\ldots, \tl{f}_k \rangle$ in $\hat{R}$. 

To this end, we will approximate $u$ by a sequence $\{u_n\}$ which we define inductively below.

Since $\{f_1,\ldots, f_k\}$ is a $v$-basis, thus 
    $u_v=\sum_{j=1}^k a_{0,j} f_{j,v}$ for some $a_{0,j} \in R$. Let $m=v_\bw(u)$ and $m_j=v_{\bw}(f_j)$. Then we have that $a_{0,j}$ is homogeneous of weight $m'_j:=m-m_j$
    
    We set \[ u_0:=\sum_{j=1}^k a_{0,j} f_{j}. \]
    Given $u_i$, we can similar write 
    $(u-u_i)_v=\sum_{j=1}^k a_{i,j} f_{j,v}$ for some homogeneous $a_{i,j} \in R$ and we set \[ u_{i+1}:=u_i+\sum_{j=1}^k a_{i,j} f_{j}. \]

By the construction, we have \[ v(a_{0,j})+v(f_j)=v(a_{0,j})+v(f_{j,v})=v(u_v)=v(u) \] and similarly
 \[ v(a_{i,j})+v(f_j)=v((u-u_i)_v)=v(u-u_i) \] if $a_{0,j}$ or $a_{i,j}$ is non-zero.

Notice  that $v(u-u_{i+1}) > v(u-v_i)$ as $(u-u_{i+1})=(u-u_i)-(u-u_i)_v$ for $i \ge 0$ and similarly $v(u-u_0) > v(u)$. It follows that $\{u_n\}$ converges to $u$. It also follows that $v(a_{i+1,j}) > v(a_{i,j})$ for all $i,j$ if $a_{i,j} \ne 0$.
Thus we may write 
$u=\sum_j \lambda_j f_j$, where $\lambda_j=\sum_{i=0}^\infty a_{i,j} \in \hat{R}$.
One can easily verify that for all $j$ 
\footnote{ If $a_{0,j} \ne 0$, then $v_\bw(\lambda_j) =v_\bw(a_{0,j}) =v_\bw(u)-v_\bw(f_j)$. Suppose that $a_{0,j} = 0$. Let $i_0$ be the smallest index $i$ such that $a_{i,j} \ne 0$. Then $v(\lambda_j) =v_\bw(a_{i_0,j}) = v_\bw(u-i_0)-v(f_j) >v_\bw(u-u_0)-v_\bw(f_j) >v_\bw(u)-v_\bw(f_j)$.},
\[ v(\lambda_j) \ge v(u)-v(f_j).\]

Therefore, \[\tl{u}=   \sum_{i=1}^k\bar{x}_1^{v(\lambda_j)+v(f_j)-v(u)} \tl{\lambda}_i\tl{f_i} \in \langle \tl{f}_1, \ldots, \tl{f}_k \rangle.\]

In the $U_i$-chart, the exceptional divisor is defined by $\bar{x}_i=0$. Hence the exceptional set of $\widetilde{X} \to X$ is defined by $ \langle \bar{x}_i, \tl{f}_1,...,\tl{f}_r \rangle $. By equation (1),  \[ \langle \bar{x}_i, \tl{f}_1,...,\tl{f}_r \rangle= \langle \bar{x}_i, \tl{f}_{1,v},...,\tl{f}_{r,v} \rangle. \]

Hence the proof is completed.
\end{proof}

\subsection{Valuations and their restrictions}

Let $R$ be a finitely generated domain over $k$. A function $v: R \to \Z_+ \cup \{\infty\}$ is said to be a valuation on $R$ (over $k$) if the following holds:
\begin{enumerate}
\item $v(x_1x_2) = v(x_1) + v(x_2)$.
\item $v(x_1+x_2) \ge \min \{ v(x_1), v(x_2) \}$.  Moreover, equality holds if $ v(x_1) \ne v(x_2)$.
\item $v(x)= \infty$ if and only if $x=0$.
\end{enumerate}

Notice that this is same as saying that $v$ can be extended to a valuation on the field of fraction of $R$, where $R$ is part of the valuation ring.

Given a valuation $v$ on $R$, we have the following ideals
\[ \fa_{v,n} := \{ x \in R | v(x) \ge n \} \]
 and they forms a filtration
\[ R=\fa_{v,0} \supset \fa_{v,1} \supset \cdots \supset \fa_{v,n} \supset \cdots \]
 Moreover,
\[\left\{ \begin{array}{l} \fa_{v,n} \cdot \fa_{v,m} \subset \fa_{v,n+m}, \text{ for all } n, m \ge 0;  \\
 \cap_{n \ge 0} \fa_{v,n} = \{ 0 \};  \\
 \text{if } xy \in \fa_{v,n}, \text{ then } x \in \fa_{v,i}, y \in \fa_{v,j} \text{ for some } i+j = n.
 \end{array} \right. \eqno{\dagger}\]

This leads to the following definition.
\begin{defn} Let $R$ be a finitely generated domain over $k$.  A filtration $\fa:=\{\fa_n \}$  on $R$ is said to be a {\it p-filtration} if $\fa_n \cdot \fa_m \subset \fa_{n+m}$ for all $n,m$. It  is said to be a {\it v-filtration} if all three conditions of $\dagger$ hold.
\end{defn}

To have better understanding of the relations between valuations and filtration. We introduced a generalized notion of valuations.

\begin{defn} Let $R$ be a finitely generated domain over $k$. A function $v: R \to \Z_{\ge 0} \cup \{\infty  \}$ is called a {\it pre-valuation} if the following holds:
\begin{enumerate}
\item $v(x_1x_2) \ge v(x_1) + v(x_2)$.
\item $v(x_1+x_2) \ge \min \{ v(x_1), v(x_2) \}$.  Moreover, equality holds if $ v(x_1) \ne v(x_2)$.
\end{enumerate}
\end{defn}

It is straightforward to check the following property. We leave the details to the readers.

\begin{prop} \label{corr}
Then there is a  one-to-one correspondence
\[ \Phi_R: \{ \text{pre-valuations of } R \} \to \{ \text{p-filtrations of } R \}\] sending a pre-valuation $v$ to its associated p-filtration $\fa_v$.
Moreover, $\Phi_R$ induced a one-to-one correspondence between valuations and $v$-filtrations of $R$. 
\end{prop}

It is easy to see that p-filtrations has nice functorial properties. 
Let $\pi: R \to S$ be a surjective $k$-homomorphism between integral domains over $k$. It is natural  to compare valuations and filtrations on $R$ and $S$. First of all, given a p-filtration $\fa:=\{\fa_n \} $ of $R$, one has a p-filtration $\pi(\fa):=\{\pi(\fa)_n \}$ such that $\pi(\fa)_n:=\pi(\fa_n)$. Similarly, given a p-filtration $\fb$ on $S$, one has a p-filtration $\pi^{-1}(\fb):=\{ \pi^{-1}(\fb_n) \}$ on $R$.

\begin{lem} \label{restriction} Let $\pi: R \to S$ be a surjective $k$-homomorphism between finitely generated domain over $k$.
Let $v_R$ be a pre-valuation on $R$. Then $v_S:=\Phi^{-1}_S  (\pi ( \Phi_R (v_R)))$ is a pre-valuation. Indeed, for any $x \in R$, we have 
\[v_S(\pi(x))=\sup\{ v_R(y)| y \in \pi^{-1} \pi(x)  \} \]
\end{lem}

\begin{proof}
First of all, $\Phi(v_R)$ is a p-filtration and hence so is $\pi \Phi(v_R)$. Then $v_S=\Phi^{-1} \pi \Phi (v_R)$ is a pre-valuation by Proposition \ref{corr}.

More explicitly, let $\fa=\Phi(v_R)$.
\[
\begin{array}{ll}
v_S(\bar{x})  & = \Psi \pi \Phi (v_R) (\bar{x}) \\
            &  = \sup\{ n | \bar{x} \in \pi(\fa_n)\} \\
            & = \sup \{n | x \in \fa_n+\ker \pi \} \\
            &= \sup \{ v_R(y)| \bar{y}=\bar{x}, y \in R \}.
  \end{array}
  \]
This completes the proof.
\end{proof}

Turning to geometric application, we usually consider valuation associated to a given weighted blowup and its restriction which we describe below.

\begin{defn}
Let $ X \subset W$ be  a  subvariety.
We have natural surjection $\pi: \Oo_W \to \Oo_X$. Moreover, let $v$ be a pre-valuation on $W$, one sees that $\Phi^{-1}_X  (\pi ( \Phi_W (v)))$ is a pre-valuation on $X$. We will denote it as $v|_X$ or $v_X$ and call it {\it the restriction of $v$ to $X$}.
\end{defn}

To justify the name of restriction, we consider the following property.

\begin{lem} \label{restr}
Let $\widetilde{W} \to W$ be a weighted blowup with exceptional divisor $\mathcal{E}$ and $\widetilde{X} \to X$ be the proper transform, which we call the weighted blowup of $X$. Suppose furthermore that $\widetilde{X}$ is normal and $E:=\mathcal{E} \cap \widetilde{X}$ is irreducible. Then the induced valuations $\mu_\mathcal{E}$ on $W$ and $\mu_E$ on $X$ satisfying $\mu_\mathcal{E}|_X= \mu_E$.
\end{lem}

\begin{proof} Let $\eta$ be the generic point (or say the prime ideal) of $\mathcal{E}$ and $\Oo_{\widetilde{W}, \eta}$ be the localization. We have a valuation $v_{\eta}$ on $\Oo_{\widetilde{W}}$ and hence  on $\Oo_W$ via the pull-back.
This gives rise to a filtration $\fa:=\Phi(v_\eta)$ on $\Oo_W$.

 Similarly, let $\bar{\eta}$ be the image of $\eta$ in $\Oo_{\widetilde{X}}$ and $\Oo_{\widetilde{X}, \bar{\eta}}$ be the localization. We have a valuation $v_{\bar{\eta}}$ on $\Oo_{\widetilde{X}}$ and on $\Oo_X$ via the pull-back.  This gives rise to a filtration $\fb:=\Phi(v_{\bar{\eta}})$ on $\Oo_X$.

 In fact, let $\mathfrak{m}$ be the maximal ideal of $\Oo_{\widetilde{W}, \eta}$, then $\fa_n=\mathfrak{m}^n$.  As the localization commutes with quotient, $\bar{\mathfrak{m}}$ is the  maximal ideal of $\Oo_{\widetilde{X}, \bar{\eta}}$ and hence $\fb_n= (\bar{\mathfrak{m}})^n = \overline{\fa}_n$.

  Note that for all $\bar{x} \in \Oo_X$,
\[ \begin{array}{ll} \mu_E(\bar{x}) &= \sup \{n | \bar{x} \in \fb_n \} = \sup\{n | \bar{x} \in \overline{\fa}_n \} \\ & = \sup\{n | y \in \fa_n, \bar{y}=\bar{x} \}=\sup \{\mu_{\mathcal{E}}(y)  |y \in \Oo_W, \bar{y}=\bar{x}  \}. \end{array} \] This completes the proof by Lemma \ref{restriction}.
 \end{proof}

\begin{defn}
An ideal $I \lhd R$ is said to be {\it prime with respect to $v$ } if $ \langle I_v \rangle$ is prime.
\end{defn}

Now we consider the restriction of valuation. 

\begin{lem}\label{restrict} Let $X \subset W$ be a subvariety defined by $I:=I_{X/W}$. Let $v=v_\bw$ be a valuation on $R_W$ associated to the weight $\bw$.  One has that
$v_X(f)=v(f)$ if and only if $f_v \not \in I_v$.
\end{lem}

 \begin{proof}
 Suppose that  $f_v \in I_v$, that is $f_v=u_v$ for some $u \in I$. Then $f':=f-u \ne 0$ has the property that $f' \equiv f  \quad (\text{mod } I)$ and $v(f') > v(f)$. Hence \[v_X(f) \ge v(f')>v(f).\]

 If $v_X(f)> v(f)$, then there exists $u \in I$ such that $v(f-u)> v(f)$. This implies that $f_v=u_v \in I_v$.
 \end{proof}

\begin{prop}\label{prime}  The restriction of valuation $v$ to $v_X$ is a valuation if and only if $ \langle I_v \rangle$ is prime.
\end{prop}

\begin{proof}
Suppose that  $ \langle I_v \rangle$ is prime. For any $\bar{f},\bar{g} \ne 0 \in R_W/I$. We pick $f', g'$  so that $f' \equiv f  (\text{mod } I)$ (resp. $g' \equiv g (\text{mod } I)$)  and $v_X(f)=v(f')$ (resp. $v_X(g)=v(g')$. Then $f'_v, g'_v \not \in I_v$ by Lemma \ref{restrict}.

Notice that a homogeneous element is contained in $\langle I_v \rangle$ if and only if it is contained in $I_v$. 
Thus,  $f'_v, g'_v \not \in \langle I_v \rangle$. It follows that $(f'g')_v=f'_v \cdot g'_v \not \in \langle I_v \rangle$. Therefore, one has
\[v_X(fg)=v_X(f'g')=v(f'g')=v(f')+v(g')=v_X(f)+v_X(g),\]
as $fg \equiv f'g' (\text{mod } I)$.

Conversely, suppose that $v_X$ is a valuation but $\langle I_v \rangle$ is not prime. There exist $f, g  \not \in \langle I_v \rangle$ but $fg \in \langle I_v \rangle$. 
We may write $f=\sum_{k=0}^\infty f_{v=k}$ and set $m_0=\min\{k| f_{v=k} \not \in I_v\}$. Let $f^-:=\sum_{k=0}^{m_0-1} f_{v=k}$ and write $f=f^-+f^+$.  Note that $f^- \in \langle I_v \rangle$ and hence $f^+ \not \in \langle I_v \rangle$ and $f^+_v \not \in I_v$. We define $g^+$ similarly so that $g^+ \not \in \langle I_v \rangle$ and $g^+_v \not \in I_v$.

 Therefore, $v_X(f^+)=v(f^+)$, $v_X(g^+)=v(g^+)$ and hence
\[ v_X(f^+g^+)=v_X(f^+) + v_X(g^+) = v(f^+)+v(g^+)=v(f^+g^+). \]
However, $fg \in \langle I_v \rangle$ implies that $f^+g^+ \in \langle I_v \rangle$ and thus $(f^+g^+)_v=f^+_v \cdot g^+_v \in \langle I_v \rangle$. It follows that $v_X(f^+g^+) > v(f^+g^+)$, which is the required contradiction.
\end{proof}

\begin{defn} Let $W$ be an affine space and let $\pi_\bw: \widetilde{W} \to W$ be a weighted blowup with weight $\bw$.  Let $\widetilde{X}$ be the proper transform of of $X$ in $\tl{W}$.
We say that the induced map $\pi_\bw|_{\widetilde{X}}:  \widetilde{X} \rightarrow X$ is a weighted blow-up of $X$ with weight $\bw$.
We usually abuse the notation of $\pi_\bw|_{\widetilde{X}}$ with $\pi_\bw$.
\end{defn}

We summarize some useful criterion
\begin{cor} Let $W$ be an affine space with a subvariety $X$. Fix a weight $\bw$ and consider the restriction of weighted blowup $\pi_\bw: \widetilde{X} \to X$. Then the following are equivalent. 
\begin{enumerate}
    \item The exceptional divisor of $E \subset \widetilde{X}$ is irreducible.
    \item The restriction $v_\bw|_X$ is a valuation on $\Oo_X$. 
    \item There is a $v_\bw$-basis $\{f_1,\ldots, f_k \}$ of $I_{X/W}$ such that $\langle f_{1,v_\bw},\ldots, f_{k, v_\bw}  \rangle $ is a prime ideal. 
    
\end{enumerate}
\end{cor}

\section{tilting and isomorphism of weighted blowups}
The purpose of this section is to introduce the notion of {\it tilting of weighted blowups}. As an application, we show that any threefold divisorial contraction to a curve can be realized as a weighted blowup.

\begin{lem} \label{tilting}
Let $R$ be a domain over $k$,  $v$ be a pre-valuation on $R$, and $\fa$ be its associated p-filtration.
Fix an integer $a \ge 0$, we can define a pre-valuation $v^\sharp$ on $R[x]$ or $R[[x]]$ by
\[v^\sharp(\sum_{i  \ge 0} c_i x^i)= \min\{  v(c_i)+ia \}_{i \ge 0}.\]
Then its associated p-filtration $\fa^\sharp$ is given by $\fa^\sharp_n= \sum_{i+aj \ge n} \fa_i[x] \cdot (x^j) $.

Pick  any $s \in R$. We consider the surjective ring homomorphism $\pi: R[x] \to R$ that maps $x$ to $s$. Let $\fa^+:=\pi(\fa^\sharp)$ be the associated p-filtration of $v^+:=v^\sharp|_R$. Then the following holds.
\begin{enumerate}
\item   $\fa^+_{n} \supset \fa_n$ for all $n$. We may write   $v^+ \succeq v$ for their associated pre-valluations.

\item If $a > v(s)$,  then $v^+ \ne v$.
\item If $a \le v(s)$, then $v^+=v$.
\end{enumerate}
\end{lem}

\begin{proof}One can check that $v^\sharp$ is a pre-valuation directly by definition. It is also straightforward to determine $\fa^\sharp$.
 Note that  $v^\sharp|_R(u)=\sup\{v^\sharp(y)| y \in R[x], \pi(y)=u\} \ge v^\sharp(u)=v(u)$. This proved the first statement.

  Indeed, if $a >v(s)$, then $v^\sharp|_R(s) \ge v^\sharp(x)=a > v(s)$, which is the second statement.

 For the last statement, suppose on the contrary that $a \le v(s)$ but $v^\sharp|_R(u) > v(u)$ for some $u\in R$. That is, there is $y=\sum c_i x^i$ so that $u=\sum c_i s^i$ and $v^\sharp(y)>v(u)$. Thus
 \[ v^\sharp(y)=\min\{ v(c_i)+ia \}_{i \ge 0} \le \min\{ v(c_i)+i v(s) \}_{i \ge 0} \le v(u), \] leads to the required contradiction.
\end{proof}

\begin{defn}
We call the pre-valuation $v^+$ defined above is a {\it tilting} of $v$ by $(s,a)$  if $a >v(s)$. Likewise, we call the p-filtration $\fa^+$ defined above is a {\it tilting} of $\fa$ by $(s,a)$.
\end{defn}

\begin{thm}\label{approx} Let $R$ be a Noetherian integral domain over $k$. Given two p-filtration $\fa$ and $\fb$ such that $ \fb  \succeq \fa$ and $\fb_1=\fa_1$.
Suppose that the graded ring $R^*:=\oplus_{n \ge 0} \fb_n/\fb_{n+1}$ is a finitely generated $R/\fb_1$-algebra. Then $\fb$ can be approximated by finitely many tilting of $\fa$.
\end{thm}

\begin{proof}
Since $R^*$ is finitely generated, there exists $N$ such that $R^*_n$ is generated by pieces of smaller degree for $n >N$. In particular,  one has that $\sum_{i+j=n+1} \fb_i/\fb_{i+1} \times \fb_j/\fb_{j+1} \to \fb_{n+1}/\fb_{n+2}$ is surjective for $n \ge N$.
In particular,
\begin{equation}
\sum_{i+j=n+1} \fb_i \cdot\fb_j =\fb_{n+1}
\end{equation}
for $n \ge N$.

Suppose that $\fb_2  \supsetneq \fa_2$.
Pick $s \in \fb_2-\fa_2$. Then $a_1:=v_{\fb}(s) \ge 2 > v_{\fa} (s)$. We consider the tilting $\fa^+$ of $\fa$ by $(s, a_1)$.
By construction, we have that $\fa^{+} \succeq \fa$ and $ \fa^+_2 \supsetneq \fa_2$. Moreover,
for any $y=\sum c_i x^i$ such that $u=\sum c_i s^i=\pi(y) \in R$. Then
 \[\begin{array}{ll}
 v^+(u) \le  v^\sharp(y) &=\min\{ v_{\fa}(c_0), v_{\fa}(c_1)+a_1,\ldots, v_{\fa}(c_n)+na_1 \}  \\
               & \le \min\{ v_{\fb}(c_0), v_{\fb}(c_1)+v_{\fb}(s) ,\ldots, v_{\fb}(c_n)+n v_{\fb}(s) \} \le v_{\fb}(u).
  \end{array}\]
 Hence $v^+(u) \le v_{\fb}(u)$ for all $u \in R$ and it follows that $\fb \succeq \fa^+$. Inductively, we have an ascending chain of filtrations
 \[ \fb  \succeq \cdots \succeq \fa^{++} \succeq \fa^+ \succeq \fa\]
 yield an ascending chain of ideals $ \fb_2  \supsetneq \cdots \supsetneq  \fa^+_2 \supsetneq \fa_2$. Since $R$ is Notherian, we see that $\fb_2=\fa^{+m_2}_2$ after $m_2$-steps of tilting for some $m_2$.

 Inductively, after $M:=m_2+\cdots + m_N$-steps of tilting, we have a p-filtration $\fa^{+M}$ such that $\fb \succeq \fa^{+M} \succeq \fa$ and $\fb_n = \fa^{+M}_n$ for all $n \le N$.
 Then by  (\ref{fg}),
\[ \sum_{i+j=N+1} \fa^{+M}_i \cdot  \fa^{+M}_j \subset \fa^{+M}_{N+1} \subset \fb_{N+1} =  \sum_{i+j=N+1} \fb_i \cdot \fb_j=  \sum_{i+j=N+1} \fa^{+M}_i  \cdot \fa^{+M}_j.\]
It follows that $\fa^{+M}_{N+1} = \fb_{N+1}$. By the same argument inductively, one has $\fa^{+M}_{m} = \fb_{m}$ for all $m >N$. Hence $\fb=\fa^{+M}$.
This completes the proof.
\end{proof}

\begin{prop} \label{fg}
Let $p \in X$ be a germ of algebraic variety and
 $\phi: Y \to X\ni p $ be a divisorial contraction.
Assume that both $Y$ and $X$ has at worst klt singularities such that $X$ is Gorenstein. 	Let $E$ be the exceptional divisor.
Then   $R_C:=\bigoplus_{m\geq0} H^0(E,-mE|_E)$ is a finitely generated $\Oo_C$-algebra and 	$R_X:=\bigoplus_{m\geq 0}\phi\ts\Oo_Y(-mE)/\phi\ts\Oo_Y(-(m+1)E)$ is a finitely generated graded $\Oo_X$-algebra.
\end{prop}	

\begin{proof}
 For any $m\geq 0$ consider the exact sequence
	\[ 0\rightarrow \Oo_Y(-(m+1)E)\rightarrow \Oo_Y(-mE)\rightarrow \Oo_E(-mE)\rightarrow 0.\]
	Since $R^1\phi\ts\Oo_Y(-(m+1)E)=R^1\phi\ts\Oo_Y(K_Y-(m+1+a(E,X))E-\phi\st K_X)=0$ by the Kawamata-Viehweg vanishing theorem, it follows
	that \[ \phi\ts\Oo_E(-mE)\cong \phi\ts\Oo_Y(-mE)/\phi\ts\Oo_Y(-(m+1)E).\] After restricting to an affine neighborhood of $p\in X$, we may assume that
	$\phi\ts\Oo_E(-mE)$ is generated by global sections. Hence we have a surjection
	\[ H^0(E,-mE|_E)\otimes\Oo_X \cong H^0(X,\phi\ts\Oo_E(-mE))\otimes\Oo_X\rightarrow \phi\ts\Oo_E(-mE)\] which induces a surjection
	\[R_C \otimes \Oo_X =\left(\bigoplus_{m\geq0} H^0(E,-mE|_E)\right)\otimes\Oo_X\rightarrow\bigoplus_{m\geq 0}\phi\ts\Oo_Y(-mE)/\phi\ts\Oo_Y(-(m+1)E)=R_X.\]
	We need to show that $R_C$ is finitely generated as a graded $\Oo_C$-algebra, which implies that
	$R_X$ is a finitely generated graded $\Oo_X$-algebra.\par
	
		To this end, let $R_{(n)}=\bigoplus_{i\geq0} H^0(E,-inE|_E)$. We can find
	$n$ so that $-nE|_E$ is very ample over $C$ and so that $R_{(n)}$ is a finitely generated $\Oo_C$-algebra. On the other hands, we can choose
	an ample divisor $H$ on $X$ such that $D=-E+\phi\st H$ is ample. It follows that $mD$ is base-point free for all $m\gg1$ by the base-point
	free theorem. Hence $\big|-mE|_E\big|=\big|-mD|_E\big|$ is base-point free for all $m\gg1$. Thus we can find $n'$ such that $n$ and $n'$
	are relatively prime, and $-n'E|_E$ is very ample, hence the graded ring $R_{(n')}$ is also a finitely generated $\Oo_C$-algebra.\par
	Now Lemma \ref{fg} yields that \[H^0(E,-anE|_E)\otimes H^0(E,-bn'E|_E)\rightarrow H^0(E,-(an+bn')E|_E)\]
	is surjective for $a$, $b\gg1$. It follows that the subring generated by $R_{(n)}$ and $R_{(n')}$
	is finitely generated and contains all but finitely many homogeneous components of $R$. Thus $R_C=\bigoplus_{m\geq0} H^0(E,-mE|_E)$ is
	a finitely generated $\Oo_C$-algebra. This finishes the proof.
\end{proof}

\begin{defn}
	Let $X$ be a variety and let $\mu$ be a pre-valuation on $X$. We say that an orthogonal valuation $(W,v)$ represents $\mu$ if there exists an embedding $X\hookrightarrow W$ such that:
	\begin{enumerate}
         \item $v_X=\mu$, and moreover
	\item there exists an affine subspace $W_0$ containing $X$ so that $v(u)=0$ for $u \ne 0 \in \Oo_{W_0}$. 
	\end{enumerate}
 
We will fix  $W_0$ to be  the smallest affine subspace of $W$ containing $X$ with above property and call it the {\it core of $X$}.

\end{defn}

\begin{thm} \label{wbu}
Let $p \in X$ be a germ of Gorenstein singularity and  $\phi: Y \to X \ni p$ be a divisorial contraction with exceptional divisor $E$.  Then $\phi$ can be realized as a weighted blowup.

More precisely, there is an embedding $X  \hookrightarrow W_0$ into a $4$-dimensional affine space $W_0$ and a further embedding $W_0  \hookrightarrow W$ into an affine space $W$. On $W$,  there is an orthogonal valuation represents  $\mu_E$, where $\mu_E$ is the valuation on $X$ induced by $E$.
\end{thm}

\begin{proof}
We first embed $X$ into a $4$-dimensional affine space $W_0$ and let $\pi_1: \Oo_{W_0} \to \Oo_X$ be the surjection of coordinate rings. For any element $u \in \Oo_{W_0}$, we denote its image in $\Oo_X$ by $\bar{u}$.  On $\Oo_{W_0}$, we pick a generating set  $\{s_1,...,s_k\}$ of $I_{Z/W_0}$, where $Z$ is the center of $E$ on $X$.  We consider $R:=\Oo_W[y_1,...,y_k]$ and a valuation $v$ on $R$ by setting
\[v(y_i) = \left\{ \begin{array}{ll} \mu_E(\bar{s}_i), &  \text{ if }  \bar{s}_i \ne 0; \\ 1, &   \text{ if }  \bar{s}_i= 0
\end{array}
\right. \eqno{\dagger}
\]
and setting $v(c)=0$ for $c \ne 0 \in \Oo_{W_0}$, $v(\sum_I c_I y^I)= \min_I \{ v(y^I)| c_I \ne 0 \}$ in general.

We consider the surjective homomorphism $ \pi: R \stackrel{\pi_2}{\to} \Oo_{W_0} \stackrel{\pi_1}{\to} \Oo_X$ which send $y_i \mapsto s_i \mapsto  \bar{s}_i$.
Let $\fa^\sharp$ be the associated v-filtration on $R$ and $\fa$ be the image of $\fa^\sharp$ on $\Oo_X$. Let $\fb$ be the  v-filtration on $\Oo_X$ associated to $\mu_E$.

We need to verify that $\fb \succeq \fa$ and $\fb_1=\fa_1$.
First note that $\fa^\sharp_1=(y_1,...,y_k)$ and hence  $\fa_1=( \bar{s}_1,..., \bar{s}_k)$.
Notice also that $\mathfrak{k}:=\ker( \Oo_{W_0} \to \Oo_X) \subset I_{Z/W}$ and hence $\fb_1=I_{Z/X}=\pi_1( I_{Z/W_0})=\fa_1$.

Also, for any $\bar{u} \in \Oo_X$ such that $\bar{u}=\pi(y)$ for some $y \in R$. We write $y=\sum c_I y^I$ with $c_I \in \Oo_W$  and hence $\bar{u}=\sum \bar{c}_I \bar{s}^I$. One has
\[ \begin{array}{ll} \mu_E(\bar{u})   &\ge \min \{ \mu_E (\bar{c}_I \bar{s}^I)| \bar{c}_I \bar{s}^I \ne 0 \} \ge \min \{ \mu_E (\bar{s}^I))| \bar{c}_I \bar{s}^I \ne 0 \} \\
&=\min \{ v  (y^I)| \bar{c}_I \bar{s}^I \ne 0 \} \ge \min\{ v(y^I)| c_I \ne 0\} =v(y).
\end{array}\]
Therefore,
$\mu_E(\bar{u}) \ge \sup\{ v(y) | \pi(y)=\bar{u}\}=v|_X(\bar{u})$. That is, $\fb \succeq \fa$.

By Proposition \ref{fg}, $R_X$ is finite generated $\Oo_X$-algebra. Together with Theorem \ref{approx}, we are done.
\end{proof}

\subsection{group action on valuations}
Let $G$ be a finite abelian group acting on $X$ (hence acting on $\Oo_X$. We are interested in valuations and filtrations compatible with the group action. In general, let $R$ be an integral domain over $k$ with a $G$-action. We have $R= \oplus_{\chi \in \hat{G}} R_\chi$ and $\fa_n= \oplus_{\chi \in \hat{G}} \fa_{n,\chi}$ for all $n$, where $\hat{G}$ denotes the character group of $G$. It follows that  $\fa_{n,\chi_1} \cdot \fa_{m, \chi_2} \subset \fa_{n+m, \chi_1 \chi_2}$ for all $n,m \ge 0$ and all  $\chi_1, \chi_2 \in \hat{G}$.

\begin{defn}
Suppose that an abelian group acts on $R$. A p-filtration is said to be  compatible with $G$, or to be a $G$-p-filtration, if $\sigma (\fa_n) = \fa_n$ for any $\sigma \in G$ and any $n$.
\end{defn}

We will need the notion of $G$-tilting.
Suppose that $s \in R_{\chi_0}$ for some $\chi_0 \in \hat{G}$. We define $R[x]_\chi:=\sum_{i \ge 0} R_{(\chi \chi_0^{-i})} x^i$. Then $R[x]=\oplus_{\chi \in \hat{G}} R[x]_\chi$ is a $G$-graded ring, also $R[x] \to R$ sending $x \to s$ is a $G$-graded morphism.

Given a p-valuation on $R$, we say that it is a $G$-p-valuation if its associated filtration is a $G$-p-filtration.
Suppose that there are $G$-p-filtration $\fb_\chi \succeq \fa_\chi $ and $\fb_{1,\chi}=\fa_{1,\chi}$ for all $\chi$. We can also approximate $\fb$ by $\fa$  in a similar manner. The first difference is that we need to approximate by  $G$-tilting each time. 

Actually, it is possible to work out a more refined version of approximation of the $G$-p-valuations so that each coordinate is $G$-semi-invariants and each tilting is $G$-compatible. The rest of this section is devoted to this refined approximation. 

Let $W$ be an affine space with coordinated $\bx=(x_1,\ldots, x_n)$ with weights $\bw=(a_1,\ldots, a_n)$.
Let $v_\bw$ be the associated orthogonal valuation. 
We will use the  notation ${\bx}(\hat{x}_i):=(x_1,\ldots, \hat{x}_i, \ldots,  x_{n})$ and ${\bw}(\hat{a}_i):=(a_1,\ldots, \hat{a}_i, \ldots, a_{n})$.

\begin{defn}

An element  $u\in\Oo_W$ is said to be a {\it twister of $v_\bw$} if  $\bx':=({\bx}(\hat{x}_i), u)$ forms a coordinate system of $W$ for some $i$,  and together with weights  $\bw':=({\bw}(\hat{a}_i),b)$ for some $b \ge 0$, then one has $v_{\bw}=v_{\bw'}$.

Let $\bar{W}$ to be the affine subspace with coordinates ${\bx}(\hat{x}_i)$ and 
$ \bar{v}:=v_{{\bw}(\hat{a}_i)}$. The we call $(\bar{W}, \bar{v})$  to be the {\it complement of $u$} (with respects to $(W,v_\bw)$).

An orthogonal valuation $(W,v_\bw)$ which represents $\mu$ is said 
to be minimal if there is no complement represents $\mu$. 
\end{defn}

In particular, for any twister $u$ of $v_\bw$ such that $X$ is contained in the complement $\bar{W}$, one has that $\bar{v}_X(u|_X)<\mu(u|_X)$.

Let $(W,v)$ be an orthogonal valuation which represents a pre-valuation $\mu$ on a variety $X$. Assume that $u$ is a twister of $v$ such that $v(u)>0$, then the complement of $u$ contains $X$.

We may write ${\bx}(\hat{u})$ instead of  ${\bx}(\hat{x}_i)$ to emphasize that it is the complement of $u$.

\begin{defn}
    Let $\mu$ be a pre-valuation on a variety $X$. Let $\mathfrak{a}$ be its associated filtration. The center of $\mu$ on $X$, denoted $C_X(\mu)$, is the subvariety defined by the ideal $\mathfrak{a}_1$. 
\end{defn}

\begin{lem}\label{vy}
	Let $\mu$ be a pre-valuation on a variety $X$ such that every component of $C_X (\mu)$ passes through the origin  $o\in X \subset W$. 
 
 Suppose  that $(W,v)$ is a minimal orthogonal valuation representing $\mu$ with core $W_0$.  Let $u$ be a twister of $v$ not contained in the core $W_0$. Then $v(u)=\mu(u|_X)>0$. 
 
 Furthermore, let $u' \in\Oo_W$ such that $u'|_X=u|_X$ and $v(u')=v(u)$. Then, $u'=u+\eta$ for some  $\eta \in u \cdot \mathfrak{m} + k[[ {\bx}(\hat{u}) ]]$, where $\mathfrak{m}$ is the maximal ideal of the origin in $W$.
\end{lem}
\begin{proof} Given $u'$, we consider its expansion in $u$. So we may write $u'=\lambda u +\xi $ for some $\lambda \in k[[{\bx}(\hat{u}), u]]$ and $\xi \in k[[{\bx}(\hat{u})]]$. 

 We first claim that $\lambda-1 \in \mathfrak{m}$. Assume on the contrary that the function $\lambda-1 \not \in \mathfrak{m}$. Then $v(u)= v((\lambda-1)u)$ and $\bar{v}(u)= \bar{v}((\lambda-1)u)$ as both $v$ and $\bar{v}$ are orthogonal valuations.  
Then
	\[v(u)= v((1-\lambda)u)=v(\xi)=\bar{v}(\xi) = \bar{v}(u) \leq \bar{v}_X(u|_X)<\mu(u|_X)=v(u),\]
which is a contradiction. 

Hence we can write $u'=u+(\lambda-1)u + \xi$ with the desired property. 
%
\end{proof}
\begin{rk}
	Keep the assumption of Lemma \ref{vy}. Then for all $u\in\Oo_W$ one has that $v(u)=0$ if and only if $u\in W_0$.
\end{rk}

\begin{defn}
	Let $X$ be a variety and let $G$ be a finite group which acts on $X$. Let $v$ be a pre-valuation on $X$. We say that $v$ is compatible with $G$ if $v(u)=v(\sigma(u))$ for all $u\in\Oo_X$ and for all $\sigma\in G$.
\end{defn}

\begin{lem}\label{gcv}
	Notation as above. Assume that $v$ is compatible with $G$. Then for all $u\in\Oo_X$ one has that $v(u)=\min\se{v(u_{\chi})}{\chi\in\hat{G},u_{\chi}\neq0}$.
\end{lem}
\begin{proof}
	Fix an element $u\in\Oo_X$. We set  $S(u) := \se{\chi\in\hat{G}}{u_{\chi}\neq 0}$. We will prove by induction on $|S(u)|$. 
 
 If $|S(u)|=1$, then $u=u_{\chi}$ for some $\chi\in\hat{G}$ and there is nothing to prove. In general, we fix $\chi_1 \ne  \chi_2 \in S(u)$ such that $v(u_{\chi_1})\geq v(u_{\chi'}) \ge v(u_{\chi_2})$ for all $\chi' \in S(u)$. Then we pick  $\sigma\in G$  such that $\chi_1(\sigma)\neq\chi_2(\sigma)$. Then, 
\[ \begin{array}{ll} \sigma(u) &=\chi_1(\sigma)u_{\chi_1}+\chi_2(\sigma)u_{\chi_2} +\sum_{\chi'\neq \chi_1, \chi_2 } \chi'(\sigma)u_{\chi'} \\

\chi_1(\sigma) \cdot u &=\chi_1(\sigma)u_{\chi_1}+\chi_1(\sigma)u_{\chi_2} +\sum_{\chi'\neq \chi_1, \chi_2 } \chi_1(\sigma)u_{\chi'} .
 \end{array}\]
Let \[u'=\sigma(u)-\chi_1(\sigma)\cdot u  \] Then $S(u')\leq S(u)-1$. By the induction hypothesis, one has that
	\[v(u')= v( (\chi_1(\sigma)-\chi_2(\sigma) )\cdot u_{\chi_2}) = v(u_{\chi_2}) .\]
Notice also that $v(\sigma(u)) = v(u) = v(\chi_1 (\sigma) \cdot u)$ and hence $v(u') \ge v(u)$. Therefore, 
\[ v(u_{\chi_2}) = v(u') \ge v(u)  \ge \min\se{v(u_{\chi})}{\chi \in S(u)} = v(u_{\chi_2}). \] This completes the proof.
\end{proof}

\begin{lem}\label{gcp}
Let $X$ be an affine variety contained in an affine space $W$. Suppose that there is a pre-valuation $\mu$ on $X$ such that every component of $C_X(\mu)$ passes through the origin $o\in X \subset W$ and compatible with a finite group $G$ action. 

If $\mu$ is represented by an orthogonal valuation $v$ on $W$, then there exists a $G$-semi-invariant coordinate system of $W$ so that $v=v_\bw$ for some weight $\bw$.

 
\end{lem}
\begin{proof}
	Let $W_0$ be the core. We may choose a coordinate system $\bx=(x_1, \ldots , x_n)$ on $W_0$ so that  $x_i|_X$ is a $G$-semi-invariant for all $i$. Let $\by=(y_1, \ldots, y_m)$ be linear functions on $W$ such that $(\bx, \by)$ forms a coordinate system of $W$. 
    
    We assume that  $y_i$ are twisters of $v$ for all $i$. Let $(\bar{W}_i,\bar{v}_i)$ be the complement of $y_i$. 
    Fix $1\leq i\leq m$. We set  $s=y_i|_X$ and we can write $s=\sum_{\chi\in\hat{G}}s_{\chi}$.

    For any $\chi$, we pick once and for all that $u_{\chi}\in\Oo_W$ satisfying $u_{\chi}|_{X}=s_{\chi}$ and $v(u_{\chi})=\mu(s_{\chi})$. We know that \[v(u_{\chi})=\mu(s_{\chi})\geq \mu(s)=v(y_i), \]
     and the equality holds for some $\chi$ by  Lemma \ref{gcv}. 

     Let $u':=\sum_{\chi\in\hat{G}}u_{\chi}$. Then $u'|_X=y_i|_X$ and $v(u')=v(y_i)$. 
      By Lemma \ref{vy}, we know that $u'=y_i+\xi$.  Hence there exists $\chi_0$ such that $u_{\chi_0}=c y_i+\xi_0$ for some unit $c$ and $v(u_{\chi_0})=v(y_i)$. 
      
      Now we define $y'_i=u_{\chi_0}$ and define the $G$-action on $y'_i$ by $\sigma(y'_i)=\chi_0(\sigma)y'_i$.
      The coordinate system $(\bx, \by')$ is a $G$-semi-invariant coordinate system of $W$ so that $v$ is an orthogonal valuation given by $v(x_i)=0$ and $v(y'_j)=v(y_j)$. 
\end{proof}

In this situation, we say that $(\bx, \by')$ {\it generates} $v$.

\begin{lem}\label{vhom}
	Assume that $(W,v)$ is an orthogonal valuation and $G$ is a finite group acting on $W$ such that coordinates are $G$-semi-invariants. Then for any $G$-semi-invariant $u$, one has that every $v$-homogeneous part of $u$  is a $G$-semi-invariant and corresponds to the same character of $u$.
\end{lem}
\begin{proof}
Since $u$ is $G$-semi-invariant, there  exists $\chi\in\hat{G}$ such that $\sigma(u)=\chi(\sigma)u$ for all $\sigma\in G$. Therefore, if we consider the decomposition into homogeneous part $u=\sum u_m$, then clearly, one has that $\chi(\sigma)u = \sum \chi(\sigma)u_m$  is a decomposition for $\chi(\sigma)u=\sigma(u)$.

On the other hand, $\sigma(u)_m=\sigma(u_m)$. Hence
$\sigma(u_m)= \chi(\sigma)u_m$ and we are done. 
\end{proof}

\begin{pro}\label{gact}
	Let $\mu$ be a pre-valuation on a variety $X$ such that every component of $C_X(\mu)$ passes through the origin $o\in X \subset W$. Assume that $(W,v)$ is a minimal orthogonal valuation representing $\mu$. Then, there exists a sequence of embeddings
	\[X\hookrightarrow  W_0 \hookrightarrow \cdots \hookrightarrow W_{m}\]
 such that 
 \begin{itemize}
     \item $W_0$ is the core of $X$ with coordinate $\bx$;
     \item $W_{i-1} \hookrightarrow W_i$ is defined by $y_i-s_i$ for some $s_i \in k[[\bx, y_1,\ldots, y_{m-1}]]$ for all $1 \le i \le m$;
     \item $(\bx, \by)$ generates $v$.
 \end{itemize}
	Furthermore, assume that there exists a finite group $G$ which acts on $X$ so that $\mu$ is $G$-compatible. Then $G$ can be extended to an action on $W$, and all functions appear above are $G$-semi-invariants.
\end{pro}
\begin{proof} Let $W_0$ be the smallest affine space in $W$ containing $X$.
	Assume that $W=\A^{n+m}$ with coordinats $\bx=(x_1,\ldots,x_n)$ and $\by'=(y'_1,\ldots,y'_m)$ where $\bx$ is a coordinate system of  $W_0$ and $\by'$ forms generators of $v$. We may assume that $v(y'_i)\leq v(y'_{i+1})$ for all $i$. 
 
 Fix $i$ and let $\bx(\hat{y_i}')$ denotes the complementary coordinates of $y'_i$ in $W$.  
 We consider 
	\[ S_i=\se{u\in\Oo_W}{u|_{W_0}=y'_i|_{W_0}, u \in y'_i \cdot \mathfrak{m} + k[[\bx(\hat{y_i}')]] }.\]  
	By Lemma \ref{vy}, we know that $v(u)<v(y'_i)$ for any $u\in S_i$. Let $s'_i\in S_i$ be an element admitting maximal $v$-valuation. We decompose \[s'_i=s_i+s''_i\] such that  $s_i$ consists the $v$-homogeneous part with $ v( s_i) < v(y'_i)$ and $v(s''_i)\geq v(y'_i)$. Let $y_i=y'_i-s''_i$. Since $v(s''_i)\geq v(y'_i)$, we know that $y_i$ is also a twister of $v$. Note that $(y_i-s_i)|_{W_0}=0$  for all $i$ and $(y_1-s_1=\ldots=y_m-s_m=0)\cong W_0$. Thus, one sees that  that $W_0$ in $W$ is defined by $y_1-s_1=\ldots=y_m-s_m=0$.
 
	Next, we assume that there is a finite group $G$ acting on $X$ such that $\mu$ is $G$-compatible. By Lemma \ref{gcp} we may assume that $x_1, \ldots, x_n$ and  $y'_1, \ldots, y'_m$ are all $G$-semi-invariants. We choose $s'_i$ to be an element in the following set
 \[ S^{G-si}_i=\se{u\in\Oo_W}{u|_{W_0}=y'_i|_{W_0}, u \in y'_i \cdot \mathfrak{m} + k[[{\bx}(\hat{y}'_i)]], u \text{ is $G$-semi-invariant} }.\] 
with maximum $v$-valuation. We decompose $s'_i=s_i+s''_i$ as above.  
By Lemma \ref{vhom}, we know that both $s_i$ and $s''_i$ are $G$-semi-invariants and correspond to the same character of $y'_i$. Let $y_i=y'_i-s''_i$. Then $y_i$ is a $G$-semi-invariant which corresponds to the same character as $s_i$. Hence $y_i-s_i$ is also a $G$-semi-invariant.
\end{proof}

\begin{lem}\label{tle}
	Assume that $\psi:\tl{X}\rightarrow X$ is a finite morphism which is isomorphic in codimension one. If $(X,\Delta)$ is a canonical pair, then so is $(X,\tl{\Delta})$, where $\tl{\Delta}=\psi\st\Delta$.\par  
	Moreover, if there exists an exceptional divisor $E$ over $X$ such that $\cen_XE$ do not contained in the ramification locus of $\phi$ and $a(E,X,\Delta)=0$, then $(\tl{X},\tl{\Delta})$ is not terminal. In this case, there exists an exceptional divisors $\tl{E}$ over $\tl{X}$ such that
	\begin{enumerate}[(1)]
	\item $\psi(\cen_{\tl{X}}\tl{E})=\cen_XE$.
	\item $a(\tl{E},\tl{X},\tl{\Delta})=0$.
	\item $\mu_{\tl{E}}(\psi\st u)=\mu_E(u)$ for any $u\in\Oo_X$.
	\end{enumerate}
\end{lem}
\begin{proof}
	By \cite[Proposition 2.14]{DL15}, for any exceptional divisor $\tl{F}$ over $\tl{X}$, there exists an exceptional divisor $F$ over $X$ such that
	\[ a(\tl{F},\tl{X},\tl{\Delta})+1=r(a(F,X,\Delta)+1)\] for some positive integer $r$. Since the right hand side is always not less than one, we know that $a(\tl{F},\tl{X},\tl{\Delta})\geq 0$. Hence, $(\tl{X},\tl{\Delta})$ is canonical.\par
	Now, assume that there exists an exceptional divisor $E$ over $X$ such that $\cen_XE$ is not contained in the ramification locus of $\psi$ and $a(E,X,\Delta)=0$. We consider the following construction in \cite[Proposition 5.20]{km}. Let $Y\rightarrow X$ be a birational morphism such that $E$ appears in $Y$ and let $\tl{Y}$ be the normalization of the main component of $\tl{X}\times_XY$. Then, there exists a divisor $\tl{E}$ on $\tl{Y}$ which maps surjectively to $E$ through the induced morphism $\tl{Y}\rightarrow Y$. One can see that $\psi(\cen_{\tl{X}}\tl{E})=\cen_XE$. Moreover, since $\cen_XE$ is not contained in the ramification locus of $\psi$, we know that $\tl{Y}\rightarrow Y$ do not ramify along the generic point of $E$. Hence we have the relation $\mu_{\tl{E}}(\psi\st u)=\mu_E(u)$ for any $u\in\Oo_X$. Now, the computation in the proof of \cite[Proposition 5.20]{km} asserts that $a(\tl{E},\tl{X},\tl{\Delta})+1=a(E,X,\Delta)+1$, which implies that $a(\tl{E},\tl{X},\tl{\Delta})=0$.
\end{proof}

\begin{lem}\label{cv}
	Assume that $Y\rightarrow X$ is a divisorial contraction between $\Q$-factorial terminal threefolds which contracts a divisor $E$ to a curve $C$ passing through a non-Gorenstein point $P \in X$. Let $\psi:\tl{X}\rightarrow X$ be the canonical cover of $P \in X$. Then, there exists a divisorial contraction $\tl{Y}\rightarrow \tl{X}$ which contracts a divisor $\tl{E}$ to an irreducible component of the pre-image of $C$ on $\tl{X}$. One has that $\tl{Y}$ is terminal and $\mu_{\tl{E}}(\psi\st u)=\mu_E(u)$ for any $u\in\Oo_X$.
\end{lem}
\begin{proof}
	Let $A_Y$ be a sufficiently ample, irreducible and reduced divisor on $Y$ which intersects an exceptional curve of $Y\rightarrow X$ at at least two points. Since $Y$ has only isolated singularities, we may assume that $(Y,A_Y)$ is log smooth near the support of $A_Y$. Let $A$ be the push-forward of $A_Y$ on $X$ and let $c=\textup{ct}(X,A)$. Then $c<1$. Also, we may assume that $E$ is the only exceptional divisor over $X$ such that $a(E,X,cA)=0$ by the tie-braking trick. Let $\tl{A}=\psi\st A$. Notice that $\psi$ is \'{e}tale away from finite points. Hence, by Lemma \ref{tle}, we know that $(\tl{X},c\tl{A})$ is canonical but not terminal and there exists an exceptional divisor $\tl{E}$ over $\tl{X}$satisfies $\psi(\cen_{\tl{X}}\tl{E})=C$, $a(\tl{E},\tl{X},c\tl{A})=0$ and $\mu_{\tl{E}}(\pi\st u)=\mu_E(u)$ for any $u\in\Oo_X$. By \cite[Theorem 17.10]{flips}, there exists a birational morphism $\tl{Y}\rightarrow \tl{X}$ such that $\textup{exc}(\tl{Y}\rightarrow \tl{X})$ is exactly the divisor $\tl{E}$. Notice that $\tl{Y}$ is obtained by first taking a log resolution $V\rightarrow \tl{X}$ of $(\tl{X},c\tl{A})$ and then running $(K_V+cA_V+\epsilon F)$-MMP over $\tl{X}$ for some effective divisor $F$ which is supported on $\textup{exc}(V\rightarrow \tl{X})$ and  $\epsilon\ll 1$. Since $c<1$, the pair $(V,cA_V+\epsilon F)$ is terminal. Therefore $\tl{Y}$ is terminal.
\end{proof}

The following refined version of representing a given pre-valuation compatible with a given finite cyclic group $G$ action works not only for the purpose of proving Theorem 1.1 but also works for the purpose of classification.  

\begin{pro} \label{group}
	Let $Y\rightarrow X$ be a divisorial contraction to a curve $C$ between $\Q$-factorial terminal threefolds. Let $\tl{X}$ be the canonical cover of $X$ and let $G$ be the cyclic group such that $X=\tl{X}/G$. Assume that the pre-image of $C$ on $\tl{X}$ is irreducible and assume that $\tl{Y}\rightarrow\tl{X}$ is the divisorial contraction in Lemma \ref{cv}. If $(\tl{W},v)$ is a minimal orthogonal valuation which represents $\tl{Y}\rightarrow \tl{X}$, then there exists a sequence of embedding $\tl{X} \hookrightarrow \tl{W}_0 \hookrightarrow \cdots \hookrightarrow \tl{W}_m$ compatible with $G$-action  such that
	\begin{enumerate}[(1)]
	\item $\tl{W}_{i-1} \hookrightarrow \tl{W}_i$ is defined by a $G$-semi-invariant $y_i-s_i$ for all $1 \le i \le m$;  

       \item $\tl{W}_0$ has coordinates $\bx=(x_1,\ldots, x_4)$ which are $G$-semi-invariants and $\tl{W}_m$ has coordinates $\bx$ and $ \by=(y_1,\ldots, y_m)$ which are $G$-semi-invariants;

	\item $Y\rightarrow X$ can be obtained by the weighted blow-up respects to the embedding $X\hookrightarrow W/G$ and the valuation $v$.
	\end{enumerate}
\end{pro}
\begin{proof}
	Let $\tl{E}=\textup{exc}(\tl{Y}\rightarrow\tl{X})$. Then, $v_{\tl{E}}$ is $G$-compatible. Indeed, since $\tl{Y}\rightarrow \tl{X}$ is a divisorial contraction to a curve $\tl{C}$ and both $\tl{Y}$ and $\tl{X}$ are terminal, we know that it is generically a blow-up. That is, if $P$ is a general point on $\tl{C}$, then there exists a local coordinate system $(x,y,z)$ near $P$ such that $\tl{C}$ is defined by $(x=y=0)$, and for any $u=\sum_{i,j\geq 0}a_{i,j}(z)x^iy^j$ one has that $\mu_{\tl{E}}(u)=\min\se{i+j}{a_{i,j}(z)\neq 0}$. Let $\sigma$ be a generator of $G$. Since $\sigma(P)\in\tl{C}$ and $(\sigma(x),\sigma(y),\sigma(z))$ also form a local coordinate system near $\sigma(P)$, we know that $\mu_{\tl{E}}(u)=\mu_{\tl{E}}(\sigma(u))$. Now (1) follows by Proposition \ref{gact}. Since $\mu_{\tl{E}}(\phi\st u)=\mu_E(u)$ for any $u\in\Oo_X$ where $E=\textup{exc}(Y\rightarrow X)$, we know that the induced weighted blow-up with respects to the embedding $X\hookrightarrow W/G$ is exactly $Y\rightarrow X$.
\end{proof}

 \begin{proof}[Proof of Theorem 1.1.]
 By Theorem \ref{wbu} and Proposition \ref{group}, we see Theorem 1.1.
 \end{proof}

\section{normalization of divisorial contractions}
In this section we assume that $X$ has Gorenstein singularities. After restricting to a suitable analytic neighborhood, we may assume that there exists an embedding $X\hookrightarrow W=\A^4$.\par 
In order to realize the weighted blowup explicitly, we need to work out the realization by weighted blowup more explicitly.
The first step is to show that we can choose $\{s_1,\ldots, s_k\}$ a generating set of $I_{C/W}$ nicely so that the defining equation of $X$ in $W$ admits a nice form for the purpose of classification. 

As both $C$ and $X$ have isolated singularities, there exists an affine open set $U \subset W$ (possibly $P \not \in U$) satisfying the following:
\begin{itemize}
\item  $C_U:=C|_U$ is a smooth curve in a smooth threefold $X_U:=X|_U$.
\item Let $I_{C_X}$ be the ideal sheaf of $C$ in $X$ and $I_{C_U}$ be the ideal sheaf of $C_U$ in $X_U$. One has $I_{C_U}/I^2_{C_U}$ is an rank $2$ free $\Oo_{X_U}$-module.
\end{itemize}

We can pick $s_1,s_2$ in $I_{C/W}$ mapping to $\bar{s}_1, \bar{s}_2 \in I_{C/X}$ so that their images  via the natural restriction map
 \[ I_{C/W} \to I_{C_X} \to I_{C_U} \to  I_{C_U}/I^2_{C_U} \]
 forms a basis as free $\Oo_{X_U}$-module. It is clear that $\mu_E(\bar{s}_i) = \mu_{E \cap \phi^{-1}(U)}(\bar{s}_i|_U)=1$ for $i=1,2$.

We can then extend $\{s_1, s_2\}$ into a generating set of $I_{C/W}$, denoted $\mathcal{B}:=\{ s_1, \ldots, s_k\}$. As in the proof of Theorem \ref{wbu}, we can define an embedding $W \hookrightarrow W'=W\times\A^k_{(y_1,...,y_k)}$ so that $W$ is defined by $(y_1-s_1=...=y_k-s_k=0) $.
On $W'$, we consider an orthogonal valuation $v_{\mathcal{B}}$ defined by $v_{\mathcal{B}}(u)=0$ if $u\in \Oo_W$, and
\[v_{\mathcal{B}}(y_i) = \left\{ \begin{array}{ll} \mu_E(\bar{s}_i), &  \text{ if }  \bar{s}_i \ne 0; \\ 1, &   \text{ if }  \bar{s}_i= 0\end{array}\right.\]

After reordering and removing redundancy, we may assume that
\begin{enumerate}
\item $v_{\mathcal{B}}(s_i)\leq v_{\mathcal{B}}(s_{i+1})$
\item $s_i \not \in  \Oo_W[s_1,...,s_{i-1},s_{i+1},...,s_k]$;
\item $\bar{s}_1|_U$ and $\bar{s}_2|_U$ generate $I_{C_U}/I^2_{C_U}$ as $\Oo_{X_U}$-module.
\end{enumerate}
We call such a generating set a {\it minimal ordered generating set} of $I_{C/W}$. Notice that $X\subset W$ is defined by a principle ideal which is contained in $I_{C/W}$. Therefore, if $\bar{s}_i=0$, then $s_i$ is exactly the defining equation of $X$, and $\bar{s}_j\neq 0$ if $i\neq j$.

Next we would like to write an expansion of $x \in \Oo_W$ by $s_1,s_2$.

\begin{prop} Let $\mathfrak{m}$ be the maximal ideal of $o \in W$. Suppose that $R$ is complete with respect to $\mathfrak{m}$ (e.g. take $R$ to be the $\mathfrak{m}$-adic completion of $\Oo_W$).

Then for any $u \in I_{C/W}$, there exists $\delta_u \in R-I_{C/W}$ and positive integers $m \le M$ so that
\[ u \cdot \delta_u \equiv \sum_{m \le i+j \le M} a_{ij} s_1^i s_2^j=\colon p_u(s_1,s_2)  \quad (\text{mod } I_X \cdot R),\]
with $a_{ij} \in  R-I_{C/W}$ if non-zero.
\end{prop}

The advantage of this expression is that $\mu_E(\bar{u})$ is usually equal to $m$ if $\bar{u} \ne 0$. However,  this expression is not unique.

\begin{proof}
 Let $\mathfrak{m}$ and $\bar{\mathfrak{m}}$ be the maximal ideal of $o \in W$ and $o \in X$ respectively. Then we denote ${R}$  (resp. ${S}$) to be the $\mathfrak{m}$-adic (resp. $\bar{\mathfrak{m}}$) completion of $\Oo_W$ (resp. $\Oo_X$). The surjective homomorphism $\pi: \Oo_W  \to \Oo_X$ yields a surjective homomorphism ${\pi}: {R} \to {S}$.

 Shrinking $U$ if necessary, we may choose a $\delta_0 \in \Oo_W$ so that the restriction of $W$ to $U$ (resp. $X$ to $X_U$) yields a localization $R_U:=R_{\delta_0}$ (resp. $S_U:=S_{\bar{\delta}_0}$). Again, we have an induced surjective homomorphism ${\pi}_U: {R}_U \to {S}_U$.

 \noindent {\bf Claim.} For any $y \in I_{C/U}^{n} - I_{C/U}^{n+1}$, there is a decomposition $y \equiv q(y) +r(y) \quad (\text{mod } I_{X_U} \cdot R_U)$ satisfying:
 \begin{itemize}
 \item $q(y)= \sum_{i+j=n} b_{ij} s_1^i s_2^j \ne 0 $ for some $b_{ij} \in R_U-I_{C/U}$ if non-zero;
  \item $r(y) \in I_{C/U}^{n'}$ for some $n' > n$.
  \end{itemize}

To see this, let us consider the surjective map \[I_{C_U/U}^n \to  I_{C_U/X_U}^n \to I_{C_U/X_U}^n/I_{C_U/X_U}^{n+1}=:\mathcal{M}.\]
As $\mathcal{M}$ is free $\Oo_{X_U}$-module generated by $\{ \bar{s}_1^i \bar{s}_2^j\}_{i+j=n}$, any element $y \in I_{C_U/U}^n$ can be represented  as $\sum_{i+j=n} c_{ij} s_1^i s_2^j$ for some $c_{ij} \in R_U$ modulo $I_{X_U}$. Let \[q(y):= \sum_{ c_{ij} \not \in I_{C_U} } c_{ij} s_1^i s_2^j, \quad r(y):=\sum_{ c_{ij}  \in I_{C_U} } c_{ij} s_1^i s_2^j.\] Note that $q(y) \ne 0$ otherwise $y \in I_{C_U/U}^{n+1}$, which is a contradiction.

 Thus for any $u \in I_{C/W}$, we consider its restriction to $U$ and abuse the notation by denoting it as $u \in I_{C/U}$. We can define sequences as follows:\\
 Since $u \in I^{n_1}_{C/U}- I^{n_1+1}_{C/U}$ for some $n_1$, we set $r_1:=r(u)$ and $q_1:=q(u)$. Inductively,
 as $r_k \in  I^{n_k}_{C/U}- I^{n_k+1}_{C/U}$ for some $n_k$, we define $r_{k+1}:=r(r_k)$ and $q_{k+1}:=r(r_k)$.
 It is straight forward to see that $u= r_N + \sum_{j=1}^N q_j$ for any $N \ge 1$. Then one sees that $\{ r_N\}$ converges to $0$ and hence $ \sum_{k=1}^N q_k=\sum_{k=1}^N \sum_{i+j=n_k} b_{ij}s_1^i s_2^j$ converges to $u$.

 	Let $\mathcal{R} \subset\Z_{\geq0}^2$ be the convex hull contains all the pairs $(i,j)$ such that $b_{ij}\neq 0$. There are only finitely many pairs
	$(i_1,j_1)$, ..., $(i_n,j_n)$ which are lying on the boundary of $\mathcal{R}$. Any other pair $(i,j)$ had the property that $(i,j)> (i_k, j_k)$ for some $k$ in the sense that $i \ge i_k, j \ge j_k$ and $i+j > i_k+j_k$. Suitably collecting terms $b_{ij}s_1^is_2^j$ with $(i,j) \ge (i_k,j_k)$ we have
\[ \sum_{(i,j)\ge (i_k,j_k)} b_{ij} s_1^i s_2^j = (b_{i_k j_k} +\sum_{(i,j) > (i_k,j_k)} b_{ij} s_1^{i-i_k} s_2^{j-j_k} ) s_1^{i_k}s_2^{j_k}=: a_{i_k j_k} s_1^{i_k} s_2^{j_k}  \] such that $ a_{i_k j_k} \in R_U-I_{C/U}$. Hence over the open set $U$,
\[ u = \sum_{1 \le k \le n}  a_{i_k j_k} s_1^{i_k} s_2^{j_k}\] with $ a_{i_k j_k} \in R_U-I_{C/U}$.

As this is a finite sum, we set $m=\min\{ i_k+j_k\}$ and $M=\max\{ i_k +j_k\}$ and we can rewrite
\[ u = \sum_{m \le i+j \le M} a_{i j} s_1^{i} s_2^{j}\] over $U$.
Multiplying both sides by certain power of $\delta_0$ will conclude the proof.
\end{proof}


\begin{cor}\label{norm} Notation as above. Suppose that $X \subset W$ is a hypersurface containing $C$ and defined by $f$. Let $\mathcal{B}:=\{ s_1, \ldots, s_k\}$ be a minimal ordered generating set. Then, we may assume that
\[ f= s_3 \delta_{s_3} - p_3(s_1, s_2)\] in $\Oo_W$.
\end{cor}

\begin{proof}
    We know that $f\in \langle s_1,...,s_k \rangle$. Since $v_{\mathcal{B}}$ is a pre-valuation, $v_{\mathcal{B}}(f)\leq v_{\mathcal{B}}(u)$ for any $u\in I_X$. Since $\bar{s}_1|_U$ and $\bar{s}_2|_U$ form a free generator of $I_{C_U}/I^2_{C_U}$, no combination of $s_1$ and $s_2$ vanishes on $X$. Therefore, $f$ is a combination of $s_1$, ..., $s_k$ and at least one $s_i$ appears in the combination for some $i>3$. An element with minimal $v_{\mathcal{B}}$-value is  $s_3 \delta_{s_3} - p_3(s_1, s_2)$.
\end{proof}

\section{tilting algorithm and factoring sequence}
In this section, we are going to develop an explicit algorithm to realize a divisorial extraction as a weighted blowup. This algorithm will allow us to examine examine equations of threefold $X$ case by case and then leads to a classification, which will be worked out in the next section.

We will mainly consider divisorial contraction $\phi: Y \to X \ni o$, where $o \in X$ is a germ of Gorenstein terminal singularity.  
Even though we expect that the construction holds more generally, we will assume that the contracted curve $C$ is an irreducible curves in the first part of this section. Then we will need to assume furthermore that the curve $C$ is a smooth curve in order to have an easy and explicit algorithm.  We will leave the general case to our subsequent work. 

Suppose now that $C$ is an irreducible lci curves. Thus we may assume that $k=3$ and $I_{C/W}=\langle s_1, s_2, s_3 \rangle$ is a prime ideal. 

We would like to introduce an algorithm that leads to an explicit classification, which we call {\it the tilting algorithm}. 
Given $X \subset W \cong \bA^4$ defined by $f$. Let $R$ be the completion of $\mathcal{O}_W$ with respect to $\mathfrak{m}$. We set $R_i:=R[y_1,\ldots, y_i]$ for any integer $i>0$.

\noindent
{\bf Step 0.}
We first consider $v_3$ on $W_3 \cong \mathbb{A}^7_{\bx, y_1,y_2,y_3}$ with \[v_3(y_i) = \left\{ \begin{array}{ll} \mu_E(\bar{s}_i), &  \text{ if }  \bar{s}_i \ne 0; \\ 1, &   \text{ if }  \bar{s}_i= 0\end{array}\right.\] for $i=1,2,3$.

Let $\pi_3: R_3 \to R$  be the projection so that $\pi(y_i)=s_i$.  We pick once and for all  
$\sigma^3 \in \pi_3^{-1}(f) \subset R_3$ so that 
\[v_3(\sigma^3)=\max \{ v_3(\sigma)| \sigma \ne 0 \in \pi_3^{-1}(f)\}. \]
Note that $\sigma^3=h_3y_3-g_3$ for some $g_3 \in R_2,h_3 \in R_3$.

\noindent
{\bf Step 1.} Suppose that we have defined an embedding of $X$ into $W_i \cong \mathbb{A}^{4+i}_{\bx, y_1,\ldots,y_{i}}$ 
 defined by 
 \[ \langle y_1-s_1,y_2-s_2, y_3-s_3, y_4-\kappa_4, \ldots, y_{i}-\kappa_{i}, \sigma^{i} \rangle \]  and a weight $v_i$ on $R_i$. Let $\pi_i: R_i \to R$  be the projection so that $\pi_i(y_i)=s_i$ for $i \le 3$ and $\pi_i(y_{i})=\kappa_{i}$ for $i \ge 4$. 
 
 In which,
 \begin{itemize}
     \item $\sigma^{i} = h_{i}y_{i} + g_{i} \in R_i$ for some $g_i \in R_{i-1}, h_i \in R_i$ is a representative of $\pi_i^{-1}(f)$ with maximal $v_i$-weight;

     \item $\kappa_{i}=p_{i} y_{i-1} + q_{i} \in R_{i-1}$ for some   $q_i \in R_{i-2}$, $p_{i} \in R_{i-1}$ is $v_{i-1}$ homogeneous;

 \end{itemize}

\noindent
{\bf Notation.}
 In what follows, we will frequently write decomposition $u=u_{v_i} + u_>$ so that  $u_{v_i}$ denotes the homogeneous part of $u$ with respect to the weight $v_i$ on $R_i$.  

 We also use the notation $u_{v_i=m}$  to denote the homogeneous part of $u$ with weight $m$ with respect to weight $v_i$, or simply $u_m$ if no confusion is likely. 

Given $u \in R_i$, we may write it as $\sum \alpha_I y^I$ with $\alpha_I \in R$. Then $\gcd(u)$ denotes $gcd\{\alpha_I\}_{\alpha_I \ne 0}$, whcih is an element in $R$. 


 \noindent
{\bf Step 2.} Let $\theta_{i+1}:=\gcd(\sigma^i_{v_i})$. If $\theta_{i+1} \ne 1$ then we consider the following decomposition.  $h_{i,v_i}=\theta_{i+1}p_{i+1}$ (resp. $g_{i,v_i}=\theta_{i+1}q_{i+1}$) for some $p_{i+1} \in R_i$ 
(resp. some $q_{i+1} \in R_{i-1}$). Furthermore, we have 

\[ \begin{array}{l} 
h_i=h_{i,v_i}+h_{i, >} = \theta_{i+1}p_{i+1}+h_{i, >},  \\
g_i=g_{i,v_i}+g_{i, >} =\theta_{i+1}q_{i+1}+g_{i, >}, 
\end{array}
\] 

We set 
\[ \left\{ \begin{array}{l} \kappa_{i+1}:= p_{i+1} y_{i}+q_{i+1},\\
\rho_{i+1}:=  h_{i, >}y_{i}+g_{i, >}. \end{array}\right. \]

Note that
\[ \sigma^i= h_iy_{i}  +g_i = \theta_{i+1}\kappa_{i+1}  +\rho_{i+1}\]
so that \[m_i:=v_i(\sigma^i)=v_i(\theta_{i+1}\kappa_{i+1})= v_i(\kappa_{i+1}), \text{ and } v_i(\rho_{i+1}) > m_i.\]

We consider $W_{i+1}:= W_i \times \bA^1_{y_{i+4}}$ and an embedding $W_i \hookrightarrow W_{i+1}$ as the section defined by $y_{i+1}-\kappa_{i+1}$.
Introduce a valuation $v_{i+1}$ on $R_{i+1}$ with 
\[ \left\{ \begin{array}{ll} v_{i+1}(y_{i+1}):= v_i(\rho_{i+1}) \\
v_{i+1}(u):=v_i(u) \textrm{ for } u \in R_i
\end{array}\right.\]

In other words, we define $v_{i+1}$ as a tilting of $v_i$ by $(\kappa_{i+1}, v_i(\rho_{i+1}))$, which is a non-trivial tilting. 
We fix $\sigma^{i+1} \in R_{i+1}$ so that it has maximal $v_{i+1}$ weight among $\pi_{i+1}^{-1}(f)$. We may write it as 
\[ \sigma^{i+1}:= h_{i+1}y_{i+1}  +g_{i+1},\] for some $h_{i+1} \in R_{i+1}$ and $g_{i+1} \in R_i$.

 \noindent
{\bf Ending Point.} If $\theta_{l+1}=1$ for some $l$, then we stop.


Note that the above process does not guarantee that we will reach the ending point. 
By considering their associated p-filtrations, we can  verify that it will be terminated after finite steps. Hence it reaches an ending point and coincides with the result of the approximation of valuation in the previous section. 
\begin{prop}
The above tilting algorithm terminates in finitely many steps.  
\end{prop}

\begin{proof}
Let $\fb$ be the filtration associated to the valuation $\mu_E$ and let $\fa^i$ be the filtration associated to $v_i$ for $i \ge 3$. 
We claim that $\fb \succeq \fa^{i} $ for all $i$, $\fa^{i+1} \succneqq \fa^{i}$ for all $i$, and $\fb_1=\fa^3_1$. Then by Theorem \ref{approx}, this ascending chain of filtrations must terminate in finitely many steps. 
\end{proof}

\begin{prop} \label{reg_seq} Suppose that $\langle s_1, s_2, s_3 \rangle$ is a prime ideal so that $R/\langle s_1, s_2, s_3 \rangle$ is a UFD. 
 Then the ideal 
\[ \langle s_1, s_2, s_3, \kappa_4,\ldots, \kappa_l, \sigma^l_{v_l} \rangle \lhd R_l \] is prime.
\end{prop}

\begin{proof}
We prove by induction on $l$. Suppose first that $l=3$. 
We claim that $\sigma^3_{v_3} \not \in \langle s_1, s_2, s_3\rangle$. Suppose on the contrary that $\sigma^3_{v_3}  \in \langle s_1, s_2, s_3\rangle$. We may write
\[ \sigma^3_{v_3}=\sum_{v_3(y^I)=m_3} \alpha_I y^I =\sum_{v_3(y^I)=m_3} (\beta_{I_1}s_1+\beta_{I_2}s_2+\beta_{I_3}s_3) y^I,  \] 
where $m_3:=v_3(\sigma^3)$. Consider 
\[\tilde{\sigma}^3_{v_3}:= \sum_{v_3(y^I)=m_3} (\beta_{I_1}y_1+\beta_{I_2}y_2+\beta_{I_3}y_3) y^I, \textrm{ and } \tilde{\sigma}^3:=\tilde{\sigma}^3_{v_3} + \sigma^3_>.\]
Then one finds that 
$\pi_3(\tilde{\sigma}^3))= \pi_3 ( \sigma^3)$ but $v_3(\tilde{\sigma}^3) = v_3(\tilde{\sigma}^3_{v_3})> v_3 ( \sigma^3)$, which a contraction to our choice of $\sigma^3$ with maximal $v_3$ weight.

Indeed, we can write $\sigma^3_{v_3}=h_{h,v_3} y_3 + g_{3,v_3}$. One sees that  $v_3(\sigma^3)=v_3(y_3)$ and  $h_{3,v_3}  \in R - \langle s_1, s_2, s_3 \rangle$. 
Consider $\bar{R}_2:=R_2/\langle s_1, s_2, s_3 \rangle$, which is a UFD. We then consider  the natural surjection  $\tau:  R_2[y_3] \to \bar{R}_2[y_3]$. The image of $\sigma^3_{v_3}$ in $\bar{R}_2[y_3]$, denoted $\bar{\sigma}$, is non-zero of degree $1$ in $\bar{R}_2[y_3]$. One sees that $\bar{\sigma}$ is irreducible in $\bar{R}_2[y_3]$. It follows that \[ R_3/\langle s_1, s_2, s_3, \sigma^3_{v_3} \rangle \cong \bar{R}_2[y_3]/\bar{\sigma}\] is an integral domain. 

Suppose now that $l \ge 4$. First recall that $\sigma^{j}$ was chosen to have maximal $v_j$ weight and $\theta_{j+1} \kappa_{j+1}$ is the $v_j$-homogeneous part of $\sigma^j$. We claim that $\kappa_{j+1}$ has maximal $v_{j}$ weight as well. To see this, suppose that there is $\tilde{\kappa}_{j+1} \in \pi_j^{-1}(\pi_j (\kappa_{j+1}))$ with $v_j(\tilde{\kappa}_{j+1}) >v_j({\kappa}_{j+1}) $, then consider $\tilde{\sigma}:= \theta_{j+1} \tilde{\kappa}_{j+1} + \rho_{j+1}$. One sees that $v_j(\tilde{\sigma}) > v_j(\sigma^j)$, which is a contradiction.
Therefore, similar argument show that $\kappa_{j+1} \not \in \langle s_1, s_2, s_3, \ldots, \kappa_l \rangle$. Moreover, $\theta_{j+1}  \in R-I_{C/W}$. The similar argument as above shows that $R_j/\langle s_1, s_2, s_3, \ldots, \kappa_{j+1} \rangle$ is an integral domain. 

Lastly,  $\sigma^l$ has maximal $v_l$ weight.  
The above argument works inductively. Hence the statement follows. 
\end{proof}

\begin{prop} \label{vb} Suppose that $\langle s_1, s_2, s_3\rangle$ is prime so that $R/\langle s_1, s_2, s_3\rangle$ is a UFD. 
The set of generators 
\[ \left\{ y_1-s_1, y_2-s_2, y_3-s_3, y_4-\kappa_4,\ldots, y_{l}-\kappa_l, \sigma^l \right\}\] forms a $v_l$-basis of $I_{X/W_l}$.
\end{prop}

We will need the following
\begin{lem}\label{skew} Let $(s_1,\ldots, s_k)$ be a regular sequence for some $k \ge 2$. Suppose that $\sum_{i=1}^k \alpha_i s_i =0 $. 
Then there exists a $k \times k$ skew-symmetric matrix $Q$ such that $ \vec{\alpha} = Q \vec{s}$, where $ ^t\vec{\alpha}=(\alpha_1, \ldots, \alpha_k)$ and $^t\vec{s} = (s_1,\ldots,s_k)$.
\end{lem}

\begin{proof}[Proof of Lemma 5.4]
We prove by induction on $k$. The case $k=2$ follows from definition directly. Suppose now that $\sum_{i=1}^k \alpha_i s_i =0 $. Then  $\alpha_k s_k$ has image $0$ in $R/\langle s_1, \ldots, s_{k-1} \rangle$ and hence $\alpha_k$ has image $0$ in the quotient. That is $\alpha_k = \sum_{i=1}^{k-1} \beta_i s_i$. Thus one has $\sum_{i=1}^{k-1} (\alpha_i + \beta_i s_k) s_i=0$. We set $ ^t\vec{\alpha'}=(\alpha_1-\beta_1 s_k, \ldots, \alpha_{k-1} -\beta_{k-1} s_k)$ and $^t\vec{s'} = (s_1,\ldots,s_{k-1})$.  By induction hypothesis, one has $ \vec{\alpha'} = Q' \vec{s'}$ for some $Q'$ which is a $(k-1) \times (k-1)$ skew-symmetric matrix. Consider a $k \times k$ skew-symmetric matrix $Q$ such that $Q_{i,j}=Q'_{i,j}$ for $1 \le i,j \le k-1$ and $Q_{i,k}=-\beta_i$ for $1 \le i \le k-1$.  One sees that $ \vec{\alpha} = Q \vec{s}$.    
\end{proof}

\begin{proof}[Proof of Proposition 5.3]
Let $J= \langle s_1, s_2, s_3, \kappa_1,\ldots, \kappa_l, \sigma^l_{v_l} \rangle$. We aim to prove that $\langle I_{X/W_l, v_l} \rangle = J$. 

By Proposition  \ref{reg_seq}, it follows that that $( s_1, s_2, s_3, \kappa_4,\ldots, \kappa_l, \sigma^l_{v_l} )$ is a regular sequence. 

We will use the following notations. We write $s_{i}:=\kappa_i$ for all $i \le 4$ and then  $^t\vec{s}=(s_1,  \ldots, s_l)$. For all $i \le l$, we set $c_i:= v(s_i) \ge 0$ and  $d_i:= v(y_j)-v(s_i) > 0$. 


Given $u \in I_{X/W_l}$, we may write is as 
\[ u=  \sum_{i=1}^l a_i (y_i-s_i) + b \sigma^l. \] 
Let $M=v_l(u)$ and $c=v_l(\sigma^l)$. 

Suppose first that $v_l(a_i(y_i-s_i)) \ge M$ for all $i$ and $v(b\sigma^l) \ge M$. Then clearly, 
\[u_M = -a_{i,M-c_i} s_i + b_{M-c} \sigma^l_c \in J. \]

Suppose next that $v_l(a_i(y_i-s_i)) < M$  for some $i$ and $v_l(b\sigma^l) \ge M$. We set
\[ a''_i= \sum_{j < M-c_i}   a_{i, j}, \quad a'_i:=a_i-a'_i,  \] and set up a decomposition of $u= u'+w$ where 
\[ w= \sum_{i=1}^l a''_i (y_i-s_i), \quad u':=u-w=\sum_{i=1}^l a'_i (y_i-s_i) + b \sigma^l. \]

Note that $u'$ has the property that $v_l(a'(y_i-s_i)) \ge M$ and hence $v(u') \ge M$. If $v(u') = M$, then $u'_M \in J$ and hence $u_M=u'_M+w_M \in J$ if and only if $w_M \in J$. If $v(u')>M$, then $u_M=w_M$. Therefore,  $u_M \in J$ is equivalent to $w_M \in J$ in either case.  

Observe now that $w$ has the following properties:
for  $j \le M$, \[ w_j = \sum_{i=1}^k -a_{i, j-c_i} s_i  + \sum_{i=1}^k a_{i,j-d_i-c_i} y_i, \] and 
$w_j=u_j=0$ for $j < M$.

\noindent
{\bf Claim.} For $j < M$, there are skew-symmetric matrices $Q_j$s so that \[ \vec{a}_j = Q_j \vec{s} - \sum_{i=1}^k Q_{j-d_i} \vec{y}_i,\]
where $^t\vec{a_j}=(a_{i,j-c_1}, \ldots, a_{k, j-c_k} ) $ and $^t\vec{y}_i=(0,\ldots, y_i, \ldots, 0)$.

\begin{proof}[Proof of the Claim]  First observe that for $j \le M$, 
\[ w_j = -^t \vec{a}_j \cdot \vec{s} + \sum_{i=1}^k {^t\vec{y}_i} \cdot  \vec{a}_{j-d_i}. \]

We prove by induction on $j$. For $m:=\min\{v(a_i (y_i-s_i)) \}$, one has $w_m = -^t \vec{a}_m \cdot \vec{s} = u_m=0$. Hence $\vec{a}_m = Q_m \vec{s}$ by Lemma \ref{skew} for some skew-symmetric $Q_m$. We set $Q_j =0$ if $j <m$, then the claim holds for $j \le m$. 

Consider now 
\[ \begin{array}{ll} w_{j+1} &= -^t \vec{a}_{j+1} \cdot \vec{s} + \sum_{i=1}^k {^t\vec{y}_i} \cdot  \vec{a}_{j+1-d_i} \\
&= -^t \vec{a}_{j+1} \cdot \vec{s} + \sum_{i=1}^k {^t\vec{y}_i} \cdot ( Q_{j+1-d_i} \vec{s} - \sum_{l=1}^k Q_{j+1-d_i-d_l} \vec{y}_l ) \\
&= -^t \vec{a}_{j+1} \cdot \vec{s} + \sum_{i=1}^k {^t\vec{y}_i} \cdot  Q_{j+1-d_i} \vec{s}   - \sum_{i=1}^k\sum_{l=1}^k {^t\vec{y}_i}Q_{j+1-d_i-d_l} \vec{y}_l  \\
&=-^t \vec{a}_{j+1} \cdot \vec{s} + \sum_{i=1}^k {^t\vec{y}_i} \cdot  Q_{j+1-d_i} \vec{s}  = (-^t \vec{a}_{j+1} + \sum_{i=1}^k {^t\vec{y}_i} \cdot  Q_{j+1-d_i}) \cdot \vec{s}. \end{array} \]
If $j+1< M$, then $w_{j+1}=0$. Thus by Lemma \ref{skew}, there is skew-symmetric $Q_{j+1}$ so that
\[ - \vec{a}_{j+1} + \sum_{i=1}^k    {^tQ_{j+1-d_i}} \vec{y}_i = -Q_{j+1} \vec{s}. \]
Hence $ \vec{a}_{j+1} = Q_{j+1} \vec{s} - \sum_{i=1}^k    Q_{j+1-d_i} \vec{y}_i$ as desired. 
\end{proof}

The above computation yields that 
\[ w_M =(-^t \vec{a}_{M} + \sum_{i=1}^k {^t\vec{y}_i} \cdot  Q_{M-d_i}) \cdot \vec{s} \in J.  \]

Finally, we consider the situation that $v(b \sigma) < M$. 

\noindent
{\bf Claim.} For $j < M$, there are skew-symmetric matrices $Q_j$s together with vectors $\vec{p}_{j-c}$ so that \[ \vec{a}_j = Q_j \vec{s} - \sum_{i=1}^k Q_{j-d_i} \vec{y}_i + \sum_{l=0}^{j-m} \sigma_{c+l} \vec{p}_{j-c-l};\]
\[ b_{j-c}=^t\vec{p}_{j-c} \cdot \vec{s} - \sum^{k}_{i=1}  {^t\vec{y}_i} \cdot \vec{p}_{j-c-d_i}.\]
where $^t\vec{a_j}=(a_{i,j-c_1}, \ldots, a_{k, j-c_k} ) $ and $^t\vec{y}_i=(0,\ldots, y_i, \ldots, 0)$.

\begin{proof}[Proof of the Claim]
Again, we prove by induction. Note that $0=w_m = -^t \vec{a}_m \cdot \vec{s}  + b_{m-c} \sigma_c^l $. As $(s_1,\ldots, s_l, \sigma_c^l)$ is a regular sequence (cf. Prop. \ref{reg_seq}), one has 
$b_{m-c}= ^t\vec{p}_{m-c} \cdot \vec{s}$ for some $\vec{p}_{m-c}$. 
We can rewrite
\[ 0=w_m=-^t \vec{a}_m \cdot \vec{s}  +  \sigma_c ^t\vec{p}_{m-c} \cdot \vec{s}. \] Hence by Lemma \ref{skew},
\[\vec{a}_m = Q_m \vec{s} + \sigma_c \vec{p}_{m-c} \] for some skew-symmetric $Q_m$. We set $Q_j=0$ and $\vec{p}_{j-c}=0$ if $j < m$.

Consider now for $m \le j < M$
\[ \begin{array}{lll} w_{j+1} &=& -^t \vec{a}_{j+1} \cdot \vec{s} + \sum_{i=1}^k {^t\vec{y}_i} \cdot  \vec{a}_{j+1-d_i} + \sigma_c b_{j-c+1} +  \sum_{l=1}^{j+1-m} \sigma_{c+l} b_{j+1-c-l}  \\
&=&-^t \vec{a}_{j+1} \cdot \vec{s} + \sum_{i=1}^k {^t\vec{y}_i} \cdot  Q_{j+1-d_i} \vec{s} + \sum_{i=1}^k \sum_{l=0}^{j-m} {^t\vec{y}_i} \sigma_{c+l} \vec{p}_{j+1-c-l-d_i} \\ 
& &+ \sigma_c b_{j-c+1} + \sum_{l=1}^{j+1-m} \sigma_{c+l} ( ^t\vec{p}_{j+1-c-l} \cdot \vec{s} - \sum_{i=1}^k {^t\vec{y}_i} \cdot \vec{p}_{j+1-c-l-d_i})\\
&=& (-^t \vec{a}_{j+1} + \sum_{i=1}^k {^t\vec{y}_i} \cdot  Q_{j+1-d_i} + \sum_{l=1}^{j+1-m} \sigma_{c+l} ^t\vec{p}_{j+1-c-l}) \cdot \vec{s} \\
& & + \sigma_c( b_{j+1-c} + \sum_i^k {^t\vec{y}_i} \cdot \vec{p}_{j+1-c-d_i}) - \sigma_{j+1+c-m} \sum_{i=1}^k {^t\vec{y}_i} \cdot \vec{p}_{m-c-d_i}. \end{array}  \]

Since $m-c-d_i < m-j$, one has $\vec{p}_{m-c-d_i}=0$ and hence $w_{j+1}$ is a combination of $\vec{s}$ and $\sigma_c$. Therefore, 
$w_{j+1}=0$ and $(s_1,\ldots, s_k, \sigma_c)$ is a regular sequence then  implies that 
\[  b_{j+1-c} + \sum_i^k {^t\vec{y}_i} \cdot \vec{p}_{j+1-c-d_i} = ^t\vec{p}_{j+1} \cdot \vec{s} \]
for some $^t\vec{p}_{j+1}$. By Lemma \ref{skew}, it follows that 
 \[ - \vec{a}_{j+1} + \sum_{i=1}^k {^tQ_{j+1-d_i}} \vec{y}_i + \sum_{l=1}^{j+1-m} \sigma_{c+l} \vec{p}_{j+1-c-l} = -Q_{j+1} \vec{s}.\]
\end{proof}

The above computation also shows that $w_M$ is generated by $s_1,\ldots, s_k, \sigma_c$. Thus  $w_M \in J$ as desired.
\end{proof}

We conclude this section by showing the following main theorem. 

\begin{thm}  \label{smooth}
Suppose that $C$ is a smooth curve.  Then the tilting algorithm describe in this section realize the given divisorial extraction $Y \to X$.     
\end{thm}

\begin{proof}
    Given a divisorial extraction $Y \to X$ with exceptional divisor $E$. Let $\mu_E$ be its associated valuation. 
    Suppose that $C$ is a smooth curve, then it follows that $\langle s_1, s_2, s_3 \rangle$ is prime and $R/\langle s_1, s_2, s_3\rangle $ is UFD. 
    
    The tilting algorithm constructs a weighted blowup in $\tilde{W}_l \to W_l$, in which $X$ embeds as a complete intersection threefold. Let $Y_l$ be the proper transform of $X$ in $\tilde{W}_l$ with exceptional set $E_l$ defined by $J=\langle s_1, s_2, s_3, \ldots \kappa_l, \sigma^l_{v_l} \rangle$. Since $J$ is prime, one sees that $E_l$ is an irreducible prime divisor. 

    By Prop \ref{prime}, $v_l|_X$ is a valuation. Observe that $v_l|_X$ and $\mu_E$ coincide on  general points of $C \subset X$ by our construction. Hence  it follows that $Y_l \cong Y$.  
\end{proof}

\section{classifications of divisorial contractions}
The purpose of this section is to classify threefold divisorial contractions to curve. As we have seen in Section 3 that a divisorial contraction to a curve can be realized as a weighted blowup. In section 5, we demonstrated an explicit tilting algorithm. 
Thus, for any divisorial contraction $\phi: Y \to X$ which contracts a divisor $E$ to a curve $C$. Suppose that $C$ passing through a Gorenstein singular point $o \in X$. It is known that $o \in X$ is a hypersurface $cDV$ singularity and the local equation can be classified. Each form of local equation, we perform the tilting algorithm and hence we have a classification of possible weighted blowups over each form. For each weighted blowup, if the proper transform has at worst terminal singularities, then we get an honest divisorial contraction. 

To work out all explicit conditions on equations of $X$ so that $Y$ has at worst terminal singularities will be very cumbersome. Therefore, we prefer to work on more general conditions on equations of $X$ so that it is enough to determine the weighted blowups and the non-Gorenstein singularities upstairs (and many more other geometric properties)
as well.

Also, suppose that $C$ passing through a non-Gorenstein singularity $o \in X$. We take its local canonical cover $\widetilde{X} \to X=\widetilde{X}/G$. By Proposition \ref{group}, $Y \to X$ is realized as a weighted blowup compatible with $G$ action. Therefore, one can examine the $G$-compatibility of the classification of Gorenstein case. Notice that the possible choices of $G$ is very limited. For example, if $o \in X$ is of type $cD/G$, then $G$ can only be cyclic group of order $2$ or $3$. Therefore, such examinations are quite easy in each case.

{\bf Setup.} $X \subset W \cong \bA^4/G$. 
First we suppose that $G$ is trivial. By   Corollary \ref{norm}, it is given by \[
f= \delta \cdot  s_3 + p(s_1,s_2)=\delta \cdot s_3 + \sum_{m \le i+j \le M} a_{ij}s_1^i s_2^j. \] 
We assume that $C$ is a smooth curve. Hence we may assume that the Jacobian matrix $\frac{\partial(s_1,s_2,s_3)}{\partial(x,y,z,u)}$ has rank $3$ at the origin. After changing the coordinates, we may assume that $s_1=z$, $s_2=u$ and $s_3=y$. Note that we may freely replace $z,u$ by their linearly independent combinations. 

On the other hand, $f=0$ defines a terminal singularity at the origin $o \in X$. By collecting all terms of $a_{ij}z^iu^j$ involve $y$, we may rewrite $f$ into the following  normal form 
\[ f=h(x,y,z,u)y+ g(x,z,u)= f_2+f_3+f_{\ge 4},\]
where $f_2$ (resp. $f_3$, $f_{\ge 4}$) denotes the part with $\textrm{mult}_o=2$ (resp. $=3$, $\ge 4$). Note that $h \not \in I_C$ as $\delta \not \in I_C$.
It is said to be of type $cA$ (resp. $cAx$; $cD$; $cE$) if $f_2$ has rank $ \ge 2$ (resp. $f_2$ has rank $2$ with two square terms; $f_2$ has rank $1$  and $f_3$ is not a perfect cube;  $f_2$ has rank $1$  and $f_3$ is  a perfect cube). 

Now we can classify the normal form into the following cases:
\begin{enumerate}[I.]
    \item $x \in h$. In this case, $xy \in f$ and hence is of $cA$ type.
    \item $x \not \in h$, $z \in h$. In this case $yz \in f$ and hence is of $cA$ type. 
    \item $x \not \in h$, $xz \in g$. In this case it is of $cA$ type. 
    \item $xz \not \in f, zu \in g$. In this case it is of $cA$ type. 
\end{enumerate}
and if none of above  holds, i,e, not of $cA$ type, then we reach the following cases: 
\begin{enumerate}[I.]\addtocounter{enumi}{4}

    \item $x, z, u \not \in h$, $y \in h$ and $g_2 \ne 0$. This is of $cAx$ type.
    
    \item $x, z, u \not \in h$,  $y \in h$, $g_2=0$ and $g_3$ is not a perfect cube. In this case, it is of $cD$ type.

    \item $x, z, u \not \in h$,  $y \in h$, $g_2= 0$ and $g_3=z^3$. In this case, it is of $cE$ type.
    
    \item $\textrm{mult}_o h >1$,  $z^2, u^2 \in g$. This is of $cAx$ type.
    \item $\textrm{mult}_o h >1$,  $g_2=z^2$ is a perfect square, $f_3$ is not a cube.  This is of $cD$ type.
    \item $\textrm{mult}_o h >1$,  $g_2=z^2$ is a perfect square, $f_3$ is  a cube. This is of $cE$ type.
\end{enumerate}
 
We may write $g=\sum_{m \le k} g_{v=k}$, where $g_{v=k} =\sum_{i+j=k} a_{ij}(x) z^i u^j$ denotes the part of $g$ with weight $k$ (as we will always set $wt(z)=wt(u)=1$). 

In this setting, we will denote \[\gcd(g_{v=k}):=\gcd \{a_{i,j}(x)\}_{i+j=k} .\]


We now  classify divisorial contractions to a smooth curve.  

\noindent
{\bf Case I.} $xy \in f$. 
Let $m$ be the weight of $g$ and replace $x$ by $x\cdot$(unit) if necessary, we may decompose $f$ into that following form  
\[ f= xy + h_+ y +g_{v=m} + g_{v>m}, \]
with $h_+ \in \langle y,z,u \rangle $. Replacing $x+h_+$ by $x$, we may express \[f=xy+g_{v=m}+g_{v>m}.\]

Suppose first that $\gcd(g_{v=m})=1$. 
We have $z^iu^j \in f$ for some $i+j=m$. Then tilting algorithm stopped at the first step.  
Then the weighted blowup with weights $(0,m,1,1)$ gives a divisorial contraction. 
More explicitly, computation shows that $Y$ has a unique non-Gorenstein point on the $U_y$-chart, which is a cyclic quotient singularity of type $\frac{1}{m}(-1,1,1)$. The pre-image of $o$ consists of $l$ components $\Gamma_j$ corresponding to the linear factors of $g_{v=m}= \prod_{j=1}^l (\alpha_j z +\beta_j u)^{m_j}$, where $\alpha_i, \beta_j$ are constants. If $m_j \ge 2$, there is a Gorenstein singularity on $\Gamma_j$.

Suppose that $\gcd(g_{v=m})=x^{k}$ for some integer $ k \ge 1$. 
Let $m'$ be the minimal weight such that $\gcd(g_{v=m'})=1$. Then $\sum_{m \le k < m'} g_{v=k} = xg'$ for some $g'$ and 
\[ f= x(y+g') + g_{v=m'} + g_{v>m'}=x\bar{y}+g_{v=m'}+g_{v > m'}, \]
where $\bar{y}=y+g'$. Then this is reduced to the previous case that $g_{v=m}=1$ by replacing $m, y, g_{v=m}, g_{v>m}$ by $m', \bar{y}, g_{v=m'}, g_{v > m'}$ respectively.

Suppose that the group $G$ is non-trivial and $\gcd(g_{v=m})=1$. We  set ${\bf w}=(0,m,1,1)$ and ${\bf v}=\frac{1}{r}(r-1, 1, \alpha, r)$ for some $\alpha$ relatively prime to $r$. One sees that $e_2 \in N_{\bf w}+\Z {\bf v}$. Then 
\[ mr{\bf v}=(mr-m)e_1+{\bf w}+ (m\alpha-1) e_3+(mr-1)e_4. \]
Hence
$U_{y}$ has singularity of type $\frac{1}{mr}(1,m \alpha-1,mr-1)$, which is terminal if and only if $(m \alpha-1, mr)=1$.  

\noindent
{\bf Case II.} $x \not \in h, z \in h$. Then we may write 
\[ f= x^{k }y + zy+h_+y + g_{v=m} +  g_{v>m}, \]
for some $k \ge 2$ and $h_+ \in \langle y,z,u \rangle$.

\noindent
{\bf Case II-1.}  $\gcd(g_{v=m})=1$. 

We consider weighted blowup with weights $(0,m,1,1)$. The only possible non-Gorenstein singularity appears in $U_y$-chart, which is given by
\[ (\tilde{f}= x^{k } + zy+\tilde{h}_+ + g_{v=m} +  \tilde{g}_{v>m} =0) \subset \frac{1}{m}(0,-1,1,1),\]
which is a $cA/m$ singularity with axial weight $k$. 

Suppose that the group $G$ is non-trivial and $m \ge 3$. We  set ${\bf w}=(0,m,1,1)$ and ${\bf v}=\frac{1}{r}(\alpha, 1, r-1, r)$ or ${\bf v}=\frac{1}{r}(r, 1, r-1, \alpha)$ for some $\alpha$ relatively prime to $r$. Similarly, 
$U_{y}$ has singularity defined by $\tilde{f}$ in  $4$-dimensional quotient space of type $\frac{1}{mr}(m\alpha,1, mr-m-1, mr-1)$ or   $\frac{1}{mr}(mr,1, mr-m-1, m\alpha-1)$, which is not terminal.  

Suppose that $m=2$, then we  set ${\bf w}=(0,2,1,1)$ and ${\bf v}=\frac{1}{2}(2,1,1,1)$. Then $U_y$ is embedded in $4$-dimensional quotient space of type $\frac{1}{4}(4,1,1,1)$ which is not terminal. 

\noindent
{\bf Case II-2.}  $\gcd(g_{v=m})\ne 1$. Let $x^{k'}=\gcd(x^k, g_{v=m})$ and $g_{v=m}=x^{k'}g'$.

The tilting algorithm leads to the embedding 
\[ \left\{ \begin{array}{l}  x^{k'}t + zy+h_+y +  g_{v>m}, \\
t- (x^{k-k'}y + g') \end{array} \right.\]
into a $5$-dimensional space and 
weighted blowup with weights $\bw=(0,m,1,1, m+1)$. 

The $U_y$-chart is given by 
\[ \left\{ \begin{array}{l}  x^{k'}ty + z+\tilde{h}_+ +  \tilde{g}_{v>m}, \\
ty- (x^{k-k'} + g') \end{array} \right. \subset \frac{1}{m}(0,-1, 1,1,1),\] which is a $cA/m$ singularity of axial weight $k-k'$. 

The $U_t$-chart is given by 
\[ \left\{ \begin{array}{l}  x^{k'} + zy+\tilde{h}_+ y +  \tilde{g}_{v>m}, \\
t- (x^{k-k'}y + g') \end{array} \right. \subset \frac{1}{m+1}(0,m, 1,1,-1),\] which is a $cA/(m+1)$ singularity of axial weight $k'$.

Suppose that the group $G$ is non-trivial. We  set ${\bf w}=(0,m,1,1, m+1)$ and ${\bf v}=\frac{1}{r}(\alpha, 1, r-1, r, \beta)$ (resp. ${\bf v}=\frac{1}{r}(r, 1, r-1, \alpha, \beta)$ for some $\alpha$ relatively prime to $r$ and $\beta=(k-k')\alpha+1$ (resp. $\beta=(k-k')r+1$). 
Similar computation shows that  
$U_{y}$-chart is not terminal.

\noindent
{\bf Case III.} $x \not \in h$, $xz  \in f$. Then $m=1$ and thus we set $wt(y)=1$. Moreover, $g_{v=1}=xg'$ for some $g'$ linear in $z,u$. Replace $z$ by $g'$, we may write 
\[ f= x^{k }y +h_+y + xz +  g_+, \]
for some $k \ge 2$,$h_+ \in \langle y,z,u \rangle$, and $v(g_+) >1 $ with $x \nmid g_+$.

Let $m':=wt(h_+y+ g_{v>1})$. 
We consider an embedding 
\[ \left\{ \begin{array}{l}  xt +h_+y +  g_{v>1}, \\
t- (x^{k-1}y + z) \end{array} \right.\]
into a $5$-dimensional space and 
weighted blowup with weights $\bw=(0,1,1,1, m')$. 

The $U_t$-chart is given by 
\[ \left\{ \begin{array}{l}  x +\tilde{h}_+y +  \tilde{g}_{v>1}, \\
t^{m'-1}- (x^{k-1}y + z) \end{array} \right. \subset \frac{1}{m'}(0,1, 1,1,m'-1),\] which is a  singularity of type $\frac{1}{m'}(1,1,m'-1)$.  

Suppose that the group $G$ is non-trivial. We may set ${\bf w}=(0,1,1,1, m')$ and $v=\frac{1}{r}(r-1, \alpha, 1, r, 1)$ (resp. $v=\frac{1}{r}(r-1, r,  1,  \alpha, 1)$ for some $\alpha$ relatively prime to $r$.  

One sees that 
$U_{t}$ is  terminal of type $\frac{1}{m'r}(m' \alpha-1, 1, -1)$ (resp. $\frac{1}{m'r}(-1, 1, m' \alpha-1)$  )  as long as $(m'\alpha-1, m'r)=1$.

\noindent
{\bf Case IV.} $zu  \in f$. Then $m=2$. Replacing $z,u$ by their independent linear combinations, we may assume that $g_{v=2}=zu$ and hence we can   write 
\[ f= x^{k }y +h_+y + zu +  g_{v>2}, \]
for some $k \ge 2$ and $h_+ \in \langle y,z,u \rangle$.

We consider weighted blowup with weights $\bw=(0,2,1,1)$. The only possible non-Gorenstein singularity appears in $U_y$-chart, which is given by
\[ (\tilde{f}= x^{k } + \tilde{h}_+ + zu +  \tilde{g}_{v>2} =0) \subset \frac{1}{2}(0,1,1,1),\]
which is a $cA/2$ singularity with axial weight $k$.

Suppose that the group $G$ is non-trivial. We may set ${\bf w}=(0,2,1,1)$ and ${\bf v}=\frac{1}{r}(r, \alpha, 1, r-1)$ (resp. $v=\frac{1}{r}(\alpha, r,  1, r-1)$ for some $\alpha$ relatively prime to $r$.  

In either cases, one sees that $U_y$ is embedded into a $4$-dimensional quotient space of type $\frac{1}{2r}(2r, \alpha, 2-\alpha, 2r-2-\alpha)$ if $\alpha$ is odd (resp. $\frac{1}{2r}(2 \alpha, r, 2-r, r-2)$ if $r$ is odd) which can not be terminal. The computation for the case that $\alpha$ is (resp. $r$) is even is similar. Therefore, there is no admissible non-trivial $G$.

\noindent
{\bf Case V.} $x,z,u \not \in h$, $y^2, z^2 \in f$. Thus we have $m=2$.  We can   write 
\[ f= x^{k }y +y^2+ h_+y + z^2 + \epsilon u^2 + g_{v>2}, \]
for some $k \ge 2$,  $h_+ \in \langle y,z,u \rangle^2$ and $\epsilon=0$ or $1$.

We consider weighted blowup with weights $(0,2,1,1)$. The only possible non-Gorenstein singularity appears in $U_y$-chart, which is given by
\[ (\tilde{f}= x^{k } + y^2+\tilde{h}_+ + z^2+\epsilon u^2  +  \tilde{g}_{v>2} =0) \subset \frac{1}{2}(0,1,1,1),\]
which is a $cA/2$ singularity with axial weight $k$.

Suppose that the group $G$ is non-trivial of order $2$. We may set $\bw=(0,2,1,1)$ and $v=\frac{1}{2}(1, 1, 0, 1)$.  
One sees that 
$U_{y}$ is  terminal of type $cAx/4$ in $\frac{1}{4}(2,1,3,1)$.

However, if  the group $G$ is of order $4$, then it is not admissible.

\noindent
{\bf Case VI.} $x, z, u \not \in h$,  $y \in h$, $g_2=0$ and $g_3$ is not a perfect cube. In this case, it is of $cD$ type.  
We may write $g_3= g_{3,1}+ g_{3,2} +g_{3,3} $, where $g_{3,j}$ consists of the terms of degree $3$ and of weight $j$. Then $g_3$ is not a perfect cube if one of the following holds:
\begin{itemize}
    \item $g_{3,1} \ne 0$;
    \item $g_{3,2} \ne 0$; 
    \item $g_{3,3}$ is not a perfect cube. 
\end{itemize}

\noindent
{\bf Case VI-1.} $m=3$.  Hence $g_{3,1}=g_{3,2}=0$ and $g_3$ is not a perfect cube. Then we may assume that $z^2u \in g_{3,3}$. 

In this situation, we may write 
\[f=y^2+x^ky+ h_+y+ z^2u+g_+, \]
for some $k \ge 2$,  $h_+ \in \langle y,z,u \rangle^2$,   and $g_+ \in \langle z, u \rangle^3$. 

We consider weighted blowup with weights $\bw=(0,3,1,1)$. The only possible non-Gorenstein singularity appears in $U_y$-chart, which is given by
\[ (\tilde{f}=  y^3+ x^{k } +\tilde{h}_+ + z^2u  +  \tilde{g}_{+} =0) \subset \frac{1}{3}(0,2,1,1),\]
which could be a $cD/3$ singularity with axial weight $2$ under the extra condition that $k=2$. 

Computation shows that there is no admissible non-trivial $G$. To see this, note that  for $|G|=2$ or $3$, ${N}/N_w$ has order $6$ or $9$. However, $\tilde{f}$ is not of $cA$ type in $U_y$ chart.

\noindent
{\bf Case VI-2.} $m=2$. Suppose both $g_{3,1}=g_{3,2} = 0$.  
In this situation, we may assume that $g_{3,3} \ni z^2u$ and we have $x^2|g_{v=2}$, so $\gcd(x^ky, g_{v=2})=x^{k'}$ for some $k' \ge 2$. We may write $g_{v=2}=x^{k'} g'$ for some $g'$.  Then we have  
\[f=y^2+x^ky+ h_+y+ x^{k'}g'+ z^2u+ g_{+}, \]
for some $k \ge 2$,  $h_+ \in \langle y,z,u \rangle^2$ and $g_+ \in \langle z,u \rangle^3$.

We consider an embedding 
\[ \left\{ \begin{array}{l}  y^2+x^{k'}t +h_+y + z^2u +  g_{+}, \\
t- (x^{k-k'}y + g') \end{array} \right.\]
into a $5$-dimensional space and 
weighted blowup with weights $\bw=(0,2,1,1,3)$. 

The $U_t$-chart is given by 
\[\left\{ \begin{array}{l}  y^2t+x^{k'} +\tilde{h}_+y + z^2u +  \tilde{g}_{+}, \\
t- (x^{k-k'}y + g') \end{array} \right. \subset \frac{1}{3}(0,2, 1,1,-1).\] 

Computation shows that if $k' \ge 3$ then it is not terminal. We thus assume that $k'=2$, then the singularity at $U_t$-chart is of type $cD/3$. 

Now the $U_y$-chart is given by 
\[\left\{ \begin{array}{l}  y+x^{2}t +\tilde{h}_+ + z^2u +  \tilde{g}_{+}, \\
ty- (x^{k-2} + g') \end{array} \right. \subset \frac{1}{2}(0,-1, 1,1,1).\]

Since $x^{2}t +\tilde{h}_+ + z^2u +  \tilde{g}_{+}$ has degree $\ge 3$,  it is terminal only if
\begin{itemize}
\item if $k=2$, then there is no non-Gorenstein singularity; or
\item if $k=3$, then there is a singularity of type $\frac{1}{2}(1,1,1)$; or
\item if $k \ge 4$ and $ g'_2$ has rank $2$, then there is a singularity of type $cA/2$; or
\item if $k=4$ and $ g'_2$ has rank $1$, then there is a singularity of type $cAx/2$; or

\item if $k \ge 4$ and $ g'_2$ has rank $1$, also either  $u^2+xz^2$ or $z^2+xu^2 \in g'$, then there is a singularity of type $cD/2$; or

\item if $k =4$, $ g'_2$ has rank $1$ and $g'_3=0$, then there is a singularity of type $cE/2$. 
\end{itemize}


Computation shows that there is no admissible non-trivial $G$. To see this, note that  for $|G|=2$ or $3$, ${N}/N_w$ has order $6$ or $9$. However, the defining equation is not of $cA$ type in the $U_t$-chart.


\noindent
{\bf Case VI-3.} $m=2$. Hence $g_{3,1}=0$. Suppose that $g_{3,2} \ne 0$ and $g_{v=3} \ne 0$. We may write $g_{v=2}=xg'$ with $g'_2\ne 0$. Namely, $g' $ contains some quadratic terms in $z,u$.  
In this situation, we have  $\gcd(x^ky, g_{v=2})=x$  Then we have  
\[f=y^2+x^ky+ h_+y+ xg'+ g_{+}, \]
for some $k \ge 2$,  $h_+ \in \langle y,z,u \rangle^2$, $g_+ \in \langle z,u \rangle^3$ and $g_{v=3} \ne 0$.

We consider an embedding 
\[ \left\{ \begin{array}{l}  y^2+xt +h_+y  +  g_{+}, \\
t- (x^{k-1}y + g') \end{array} \right.\]
into a $5$-dimensional space and 
weighted blowup with weights $\bw=(0,2,1,1,3)$. 

The $U_t$-chart is given by 
\[\left\{ \begin{array}{l}  y^2t+x+\tilde{h}_+y +  \tilde{g}_{+}, \\
t- (x^{k-1}y + g') \end{array} \right. \subset \frac{1}{3}(0,2, 1,1,-1),\] which has a  quotient singularity of type $\frac{1}{3}(2,1,1)$. 

The $U_y$-chart is given by 
\[\left\{ \begin{array}{l}  y+xt +\tilde{h}_+  +  \tilde{g}_{+}, \\
ty- (x^{k-1} + g') \end{array} \right. \subset \frac{1}{2}(0,-1, 1,1,1).\]

As degree of  $\tilde{h}_+  +  \tilde{g}_{+} \ge 3$, it is terminal only if
\begin{itemize}
\item if $k=2$, then there is a singularity of type $\frac{1}{2}(1,1,1)$; or
\item if $k \ge 3$ and $ g'_2$ has rank $2$, then there is a singularity of type $cA/2$; or
\item if $k=3$ and $ g'_2$ has rank $1$, then there is a singularity of type $cAx/2$; or
\item if $k \ge 3$ and $ g'_2$ has rank $1$, then there is a singularity of type $cD/2$.
\end{itemize}

Computation shows if $G$ has order $2$, then it is not admissible. 

Suppose now that $G$ has order $3$.
Then $U_t$-chart is a terminal singularity of type $\frac{1}{9}(-2,5,2)$.
Computation on $U_y$-charts shows that 
\begin{itemize}
\item if $k=2$, then there is a singularity of type $\frac{1}{6}(1,-1,1)$;
\item the other cases can not be terminal.
\end{itemize}
However, by the classification of terminal singularities, $y$ is invariant and $x$ is not invariant under the $G$-action. Thus, $f$ containing $y^2$ and $x^2y$ can not be semi-invariant. We thus conclude that a group of order $3$ is also not admissible.

\noindent
{\bf Case VI-4.} $m=2$,  $g_{3,2} \ne 0$ (hence $g_3$ is not a perfect cube) and $g_{v=3} = 0$.

In this situation, we have $\gcd(x^ky, g_{v=2})=x$. We may write $g_{v=2}=x g'$ for some $g'$ with $g'_2 \ne 0$.  Then we have  
\[f=y^2+x^ky+ h_+y+ x g' +g_{v \ge 4}, \]
for some $k \ge 2$ and $h_+ \in \langle y,z,u \rangle^2$.

The tilting algorithm leads to the embedding 
\[ \left\{ \begin{array}{l}  y^2+xt +h_+y +g_{v \ge 4}, \\
t- (x^{k-1}y + g') \end{array} \right.\]
into a $5$-dimensional space and 
weighted blowup with weights $\bw=(0,2,1,1,4)$. 

The $U_t$-chart is given by 
\[\left\{ \begin{array}{l}  y^2+x +\tilde{h}_+y +\tilde{g}_{v \ge 4}, \\
t^2- (x^{k-1}y + g') \end{array} \right. \subset \frac{1}{4}(0,2, 1,1,-1),\] 
which is a singularity of type $cAx/4$. There is no non-Gorenstein singularity in $U_y$-chart. 

Computation shows that there is no admissible non-trivial $G$. To see this, note that  for $|G|=2$ or $3$, ${N}/N_w$ has order $8$ or $12$. However, the defining equation is not of $cA$ type in the $U_t$-chart.

\noindent
{\bf Case VI-5.} $m=1$ and $g_{3}$ is not a perfect cube.  
In this situation, we have $x^2|g_{v=1}$, so $\gcd(x^ky, g_{v=1})=x^{k'}$ for some $k' \ge 2$. We may write $g_{v=1}=x^{k'} g'$ for some $g'$.  Then we have  
\[f=y^2+x^ky+ h_+y+ x^{k'}g'+ g_{+}, \]
for some $k \ge 2$,  $h_+ \in \langle y,z,u \rangle^2$ and $g_+ \in x\langle z,u \rangle^2 + \langle z,u \rangle^3$.

We consider an embedding 
\[ \left\{ \begin{array}{l}  y^2+x^{k'}t +h_+y +  g_{+}, \\
t- (x^{k-k'}y + g') \end{array} \right.\]
into a $5$-dimensional space and 
weighted blowup with weights $(0,1,1,1,2)$.

The $U_t$-chart is given by 
\[\left\{ \begin{array}{l}  y^2+x^{k'} +\tilde{h}_+   + \tilde{g}_{+}, \\
t- (x^{k-k'}y + g') \end{array} \right. \subset \frac{1}{2}(0,1, 1,1,1).\] 

Computation shows that $\tilde{h}_+$ and $\tilde{g}_+$ have degree $\ge 4$. Therefore, one has terminal singularity only under the following situations: 
\begin{itemize}
    \item if $k'=2$, then it is of type $cAx/2$;
    \item if $k' \ge 3$ and some of $\{xz^2, xzu, xu^2\} \in g_+$, then it is of type $cD/2$;
    \item if $k'=3$, none of $\{xz^2, xzu, xu^2\} \in g_+$  and some of $\{ z^3, z^2u, zu^2, u^3\}  \in g_+$, then it is of type $cE/2$.   
\end{itemize}

Computation shows that there is no admissible non-trivial $G$. To see this, note that  for $|G|=2$ or $3$, ${N}/N_w$ has order $4$ or $6$. However, in $U_t$-chart, either the defining equation is of $cAx$ type with incorrect weights in the quotient space or the defining equation is not of $cA$ type in a quotient space by a group of order $6$. Both cases can not be terminal.

\noindent
{\bf Case VII.} $x, z, u \not \in h$,  $y \in h$, $g_2=0$ and $g_3$ is a perfect cube. We have $g_{3,1}=g_{3,2}=0$ and may assume that $g_3=z^3$. In this case, it is of $cE$ type.

\noindent
{\bf Case VII-1.} $m=3$ and $g_3=z^3$. 

In this situation, we may write 
\[f=y^2+x^ky+ h_+y+ z^3+g_{v>3}, \]
for some $k \ge 2$, and  $h_+ \in \langle y,z,u \rangle^2$. 

We consider weighted blowup with weights $\bw=(0,3,1,1)$. The only possible non-Gorenstein singularity appears in $U_y$-chart, which is given by
\[ (\tilde{f}=  y^3+ x^{k } +\tilde{h}_+ + z^3  +  \tilde{g}_{v>3} =0) \subset \frac{1}{3}(0,2,1,1),\]
which could be a $cD/3$ singularity with axial weight $2$ under the extra condition that $k=2$. 

There is no admissible non-trivial $G$.  Otherwise $N/N_w$ has order $6$ but the defining equation is not of $cA$ type in the $U_y$-chart. 

\noindent
{\bf Case VII-2.} $m=2$ and $g_3=z^3$. 
In this situation, we have $x^2|g_{v=2}$.  Then we may write 
\[f=y^2+x^ky+ h_+y+ x^{k'}g'+ z^3+g_+, \]
for some $k \ge 2$,  $h_+ \in \langle y,z,u \rangle^2$,   $g_+ \in \langle z, u \rangle^3$, and  $x^{k'}=\gcd(x^ky, g_{v=2})$ so that we write $g_{v=2}=x^{k'} g'$.

We consider an embedding 
\[ \left\{ \begin{array}{l}  y^2+x^{k'}t +h_+y + z^3 +  g_{+}, \\
t- (x^{k-k'}y + g') \end{array} \right.\]
into a $5$-dimensional space and 
weighted blowup with weights $\bw=(0,2,1,1,3)$.

The $U_t$-chart is given by 
\[\left\{ \begin{array}{l}  y^2t+x^{k'} +\tilde{h}_+y + z^3 +  \tilde{g}_{+}, \\
t- (x^{k-k'}y + g') \end{array} \right. \subset \frac{1}{3}(0,2, 1,1,-1).\] 

As $k' \ge 2$ in this case, one sees that it is terminal only when  $k'=2$. In this situation, it is a singularity of type $cD/3$.

Assuming that $k'=2$, then $U_y$-chart is given by 
\[\left\{ \begin{array}{l}  y+x^{2}t +\tilde{h}_+ + z^3 +  \tilde{g}_{+}, \\
ty- (x^{k-2} + g') \end{array} \right. \subset \frac{1}{2}(0,-1, 1,1,1).\]

Since $x^{2}t +\tilde{h}_+ + z^3 +  \tilde{g}_{+}$ has degree $\ge 3$,  it is terminal only 
\begin{itemize}
\item if $k=2$, then there is no non-Gorenstein singularity; or
\item if $k=3$, then there is a singularity of type $\frac{1}{2}(1,1,1)$; or
\item if $k \ge 4$ and $ g'_2$ has rank $2$, then there is a singularity of type $cA/2$; or
\item if $k=4$ and $ g'_2$ has rank $1$, then there is a singularity of type $cAx/2$; or
\item if $k \ge 5$ and $ g'_2$ has rank $1$, also either  $u^2+xz^2$ or $z^2+xu^2 \in g'$, then there is a singularity of type $cD/2$; or

\item if $k =5$, $ g'_2$ has rank $1$ and $g'_3=0$, then there is a singularity of type $cE/2$. 
\end{itemize}


There is no admissible non-trivial $G$.  Otherwise $N/N_w$ has order $6$ but the defining equation is not of $cA$ type in the $U_t$-chart. 

\noindent
{\bf Case VII-3.} $m=1$ and $g_3=z^3$. 
In this situation, we have $x^3|g_{v=1}$.  Let $x^{k'}=\gcd(x^ky, g_{v=1})$.  So  we have  $g_{v=1}=x^{k'} g'$ for some $g'$. We may write 
\[f=y^2+x^ky+ h_+y+ x^{k'}g'+ z^3+g_+, \]
for some $k \ge 2$,  $h_+ \in \langle y,z,u \rangle^2$, and  $g_+ \in x^2\langle z, u \rangle^2 + x\langle z, u \rangle^3 + \langle z, u \rangle^4$.

We consider an embedding 
\[ \left\{ \begin{array}{l}  y^2+x^{k'}t +h_+y + z^3 +  g_{+}, \\
t- (x^{k-k'}y + g') \end{array} \right.\]
into a $5$-dimensional space and 
weighted blowup with weights $(0,1,1,1,2)$.

The $U_t$-chart is given by 
\[\left\{ \begin{array}{l}  y^2+x^{k'} +\tilde{h}_+y + z^3t +  \tilde{g}_{+}, \\
t- (x^{k-k'}y + g') \end{array} \right. \subset \frac{1}{2}(0,1, 1,1,1).\] 

Computation shows that only if $k'=2$ (and hence $k=2$), then there is a singularity of type $cAx/2$. 

If $G$ is non-trivial, then it has order $2$ and it follows that $N/N_w$ is non-cyclic of order $4$. Thus $U_t$ is not terminal and hence there is no admissible non-trivial $G$. 

\noindent
{\bf Case VIII.} $\textrm{mult}_o h >1$,  $z^2, u^2 \in g$. 

We can   write 
\[ f= x^{k }y +h_+y + z^2 +u^2 +  g_{+}, \]
for some $k \ge 2$,  $h_+ \in x\langle y,z,u \rangle + \langle y,z,u \rangle^2$ and $g_+ \in \langle x,z,u \rangle^3$.

We consider weighted blowup with weights $\bw=(0,2,1,1)$. The only possible non-Gorenstein singularity appears in $U_y$-chart, which is given by
\[ (\tilde{f}= x^{k } +\tilde{h}_+ + z^2+u^2  +  \tilde{g}_{+} =0) \subset \frac{1}{2}(0,1,1,1),\]
which is a $cA/2$ singularity with axial weight $k$.

Suppose that $G$ has order $2$. We may choose $\bv=\frac{1}{2}(1,1,1,0)$ so that $f$ is of type $cAx/2$. Then one finds that $U_y$ chart yields a singularity of type $cAx/4$ in $\frac{1}{4}(2,1,1,-1)$.  

Suppose that $G$ has order $4$, then computation shows that it is not admissible.  

\noindent
{\bf Case IX.}  $\textrm{mult}_o h >1$,  $g_2=z^2$ is a perfect square, $f_3$ is not a cube.  This is of $cD$ type.

\noindent
{\bf Case IX-1.} $m=2$ and $x^2 \in h$. 

We can   write 
\[ f= x^{2 }y +h_+y + z^2  +  g_{+}, \]
for some   $h_+ \in  x \langle y,z,u \rangle+ \langle y,z,u \rangle^2$ and $g_+ \in x\langle z,u \rangle^2+\langle z,u \rangle^3$.

We consider weighted blowup with weights $\bw=(0,2,1,1)$. The only possible non-Gorenstein singularity appears in $U_y$-chart, which is given by
\[ (\tilde{f}= x^{2 } +\tilde{h}_+ + z^2  +  \tilde{g}_{+} =0) \subset \frac{1}{2}(0,1,1,1),\]
which is a $cAx/2$ singularity. 
 
If $G$ has order $2$, then one sees that  $N/N_w$ is non-cyclic of  order $4$. If $G$ has order $3$, one has $N/N_w$ of order $6$ but the defining equation is not of $cA$ type in the $U_y$-chart. Thus, there is no admissible non-trivial $G$.

\noindent
{\bf Case IX-2.} $m=2$, $x^2 \not \in h$, and $g_{3,2} \ne 0$. 

As $z^2 \in g$, we may assume that $g_{3,2}=xu^2+ \epsilon xuz$ for $\epsilon=0$ or $1$.
We can   write 
\[ f= x^{k }y +h_+y + z^2  + xu^2+\epsilon xuz g_{+}, \]
for some $k \ge 3$, $epsilon=0,1$, $h_+ \in x\langle y,z,u \rangle+\langle y,z,u \rangle^2$ and $g_+ \in x^2\langle z,u \rangle^2+\langle z,u \rangle^3$.

We consider weighted blowup with weights $\bw=(0,2,1,1)$. The only possible non-Gorenstein singularity appears in $U_y$-chart, which is given by
\[ (\tilde{f}= x^{k } +\tilde{h}_+ + z^2 + xu^2 +  \tilde{g}_{+} =0) \subset \frac{1}{2}(0,1,1,1),\]
which is a $cD/2$ singularity.

Similarly, there is no admissible non-trivial $G$ since $N/N_w$ has order $4$ or $6$ but the defining equation is  of $cD$ type in the $U_y$-chart. 

\noindent
{\bf Case IX-3.} $m=2$, $x^2 \not \in h$,  $g_{3,2}=0$, and $\lambda xy \in f_3$ for some  $\lambda$ linear in $y,u$. 

We can   write 
\[ f= x^{k }y +\lambda xy+ h_+y + z^2  + g_{+}, \]
for some $k \ge 3$, $\lambda$ is linear in $y,z,u$,   $h_+ \in x^2 \langle y,z,u \rangle+\langle y,z,u \rangle^2$, and $g_+ \in x^2\langle z,u \rangle^2+\langle z,u \rangle^3$.

We consider weighted blowup with weights $\bw=(0,2,1,1)$. The only possible non-Gorenstein singularity appears in $U_y$-chart, which is given by
\[ (\tilde{f}= x^{k } +\lambda xy+\tilde{h}_+ + z^2  +  \tilde{g}_{+} =0) \subset \frac{1}{2}(0,1,1,1).\]
Computation shows that $\tilde{h}_+$ and $\tilde{g}_{+}$ has degree $\ge 4$. Therefore, it is  of type 
 $cD/2$. 

Similarly, there is no admissible non-trivial $G$ since $N/N_w$ has order $4$ or $6$ but the defining equation is  of $cD$ type in the $U_y$-chart.

\noindent
{\bf Case IX-4.} $m=2$, $x^2 \not \in h$,  $g_{3,2}=0$, and $f_3=g_3$ is not a cube. 

We can   write 
\[ f= x^{k }y +h_+y + z^2  +u^3+ zu^2+ g_{+}, \]
for some $k \ge 3$,   $h_+ \in x^2\langle y,z,u \rangle+\langle y,z,u \rangle^2$, and $g_+ \in x^2\langle z,u \rangle^2+x\langle z,u \rangle^3+ \langle z,u \rangle^4$.

We consider weighted blowup with weights $\bw=(0,2,1,1)$. The only possible non-Gorenstein singularity appears in $U_y$-chart, which is given by
\[ (\tilde{f}= x^{k } +\tilde{h}_+ + z^2  +  yu^3+yzu^2+\tilde{g}_{+} =0) \subset \frac{1}{2}(0,1,1,1).\]
Computation shows that $\tilde{h}_+$ and $\tilde{g}_{+}$ has degree $\ge 4$. Therefore, it is terminal only if $k=3$, 
which is a $cE/2$ singularity. 

Similarly, there is no admissible non-trivial $G$ since $N/N_w$ has order $4$ or $6$ but the defining equation is  of $cE$ type in the $U_y$-chart.

\noindent
{\bf Case IX-5.}  $m=1$, and $f_3$ is not a cube. 

In this situation, we have $x^2|g_{v=1}$. Let $x^{k'}=\gcd(x^ky, g_{v=1})$ with $k' \ge 2$. So we write  $g_{v=1}=x^{k'} g'$ for some $g'$. 
We can   write 
\[ f= x^{k }y +h_+y + z^2  + x^{k'}g'+  g_+, \]
for some $k \ge k' \ge 2$,   $h_+ \in x\langle y,z,u \rangle+\langle y,z,u \rangle^2$ and $g_+ \in x\langle z,u \rangle^2+ \langle z,u \rangle^3$.  

We consider an embedding 
\[ \left\{ \begin{array}{l}  x^{k'}t +h_+y + z^2 +  g_{+}, \\
t- (x^{k-k'}y + g') \end{array} \right.\]
into a $5$-dimensional space and 
weighted blowup with weights $\bw=(0,1,1,1,2)$. 

The $U_t$-chart  is given by
\[ \left\{ \begin{array}{l}  x^{k'} +\tilde{h}_+y + z^2 +  \tilde{g}_{+}, \\
t- (x^{k-k'}y + g') \end{array} \right. \subset \frac{1}{2}(0,1,1,1,1).\]
Computation shows that it is terminal only 
\begin{itemize}
\item if $k'=2$, then there is a $cAx/2$ singularity;
\item if $k' \ge 3$ and at least one of $\{xy^2, xyu, xu^2\}$ in $f$, then there is a $cD/2$ singularity; 

\item if $k'=3$, none of $\{xy^2, xyu, xu^2\}$ in $f$,  some of $\{y^3, y^2z, y^2u, yzu, yu^2, zu^2, u^3\}$ in $f$ and $f_3$ is not a cube, then there is a $cE/2$ singularity.
\end{itemize}

Computation shows that there is no admissible non-trivial group acting on any of the above three subcases.

\noindent
{\bf Case X.}  $\textrm{mult}_o h >1$,  $g_2=z^2$ is a perfect square, $f_3$ is  a cube. This is of $cE$ type.

Note that $x^2 \not \in h$ and also $g_{3,1}=g_{3,2}=0$ in this case.

\noindent
{\bf Case X-1.} $m=2$ and $g_{3,3} \ne 0$. We may assume that $g_{3,3}=u^3$.  

We can   write 
\[ f= x^{k }y +h_+y + z^2  + u^3+  g_{+}, \]
for some  $k \ge 3$,  $h_+ \in x^2\langle y,z,u \rangle+x\langle y,z,u \rangle^2+\langle y,z,u \rangle^3$ and $\ g_+ \in x^2\langle z, u \rangle^2 + x\langle z, u \rangle^3 +\langle z, u \rangle^4$.

We consider weighted blowup with weights $(0,2,1,1)$. The only possible non-Gorenstein singularity appears in $U_y$-chart, which is given by
\[ (\tilde{f}= x^{k } +\tilde{h}_+ + z^2  + u^3y +\tilde{g}_{+} =0) \subset \frac{1}{2}(0,1,1,1).\]

Computation shows that $\tilde{h}_+$ and $\tilde{g}_{+}$ has degree $\ge 4$. Therefore, it is terminal only if $k=3$, 
which is a $cE/2$ singularity. 

Similarly, there is no admissible non-trivial $G$ since $N/N_w$ has order $4$ but the defining equation is  of $cE$ type in the $U_y$-chart. 

\noindent
{\bf Case X-2.} $m=2$ and $g_{3,3} = 0$. We may assume that $f_3=y^3$.  

We can   write 
\[ f= x^{k }y +y^3+ h_+y + z^2  +  g_{+}, \]
for some  $k \ge 3$,  $h_+ \in x^2\langle y,z,u \rangle+x\langle y,z,u \rangle^2+\langle y,z,u \rangle^3$ and $\ g_+ \in x^2\langle z, u \rangle^2 + x\langle z, u \rangle^3 +\langle z, u \rangle^4$.

We consider weighted blowup with weights $(0,2,1,1)$. The only possible non-Gorenstein singularity appears in $U_y$-chart, which is given by
\[ (\tilde{f}= x^{k } +y^4+\tilde{h}_+ + z^2   +\tilde{g}_{+} =0) \subset \frac{1}{2}(0,1,1,1).\]

Computation shows that $\tilde{h}_+$ and $\tilde{g}_{+}$ has degree $\ge 4$. Therefore, it is terminal only if $k=3$, 
which is a $cE/2$ singularity. 

Similarly, there is no admissible non-trivial $G$ since $N/N_w$ has order $4$ but the defining equation is  of $cE$ type in the $U_y$-chart. 

\noindent
{\bf Case X-3.}  $m=1$ and $f_3=u^3$. 

In this situation, we have $x^3|g_{v=1}$. Let $x^{k'}=\gcd(x^ky, g_{v=1})$ with $k' \ge 3$.    So  we have  $g_{v=1}=x^{k'} g'$ for some $g'$. 
We can   write 
\[ f= x^{k }y +h_+y + z^2  + u^3+x^{k'}g'+  g_{+}, \]
for some $k \ge k' \ge 3$,  $h_+ \in x^2\langle y,z,u \rangle + x\langle y,z,u \rangle^2+\langle y,z,u \rangle^3$ and $\ g_+ \in x^2\langle z, u \rangle^2 + x\langle z, u \rangle^3 +\langle z, u \rangle^4$.

We consider an embedding 
\[ \left\{ \begin{array}{l}  x^{k'}t +h_+y + z^2 + u^3+ g_{+}, \\
t- (x^{k-k'}y + g') \end{array} \right.\]
into a $5$-dimensional space and 
weighted blowup with weights $\bw=(0,1,1,1,2)$. 

The $U_t$-chart  is given by
\[ \left\{ \begin{array}{l}  x^{k'} +\tilde{h}_+y + z^2 + u^3t + \tilde{g}_{+}, \\
t- (x^{k-k'} + g') \end{array} \right. \subset \frac{1}{2}(0,1,1,1,1).\]
It is terminal of type $cE/2$ if $k'=3$. 

Suppose that $G$ has order $2$. Computation shows that $U_t$ is in $\frac{1}{4}(2,1,1,-1,1)$ or $\frac{1}{4}(2,-1,1,1,1)$, which can not be terminal.

\noindent
{\bf Case X-4.}  $m=1$ and $f_3=y^3$. 

Similarly, we can   write 
\[ f= x^{k }y +h_+y+y^3+ z^2  +x^{k'}g'+  g_{+}, \]
for some $k \ge k' \ge 3$,  $h_+ \in x^2\langle y,z,u \rangle + x\langle y,z,u \rangle^2+\langle y,z,u \rangle^3$ and $\ g_+ \in x^2\langle z, u \rangle^2 + x\langle z, u \rangle^3 +\langle z, u \rangle^4$.

Same computation as in Case X-3 gives a terminal singularity of type $cE/2$ if $k'=3$. 

There is also no admissible non-trivial group $G$. 


\begin{proof}[Proof of Theorem \ref{thm2}]
The classifications of divisorial contractions are summarized in the following Tables.    
\end{proof}

\begin{proof}[Proof of Theorem \ref{gethm}] One can pick the surface defined by $u=0$ for No. $1,2,3, 23, 30, \ldots ,37$ and $Q3$; $y=0$ for No. $4, 5, 6, 7,  24, 26$ and $Q4$; $y+\mu u$ for remaining cases, where $\mu$ is a general element in $\mathbb{C}$.  
\end{proof}

{\Small

\begin{table}[h] 
    \begin{tabular}{|l|l|l|l|l|}
    \hline
    {\bf No.}&$X\hookrightarrow  W$ and description & $\begin{array}{l} I_{C/W} \textup{ and }\\\textup{weights }{\bf w}\end{array}$ & $\begin{array}{l}  \textup{Non-Gor. }\\\textup{sing. of }Y\end{array}$ &  {\bf Ref.}\\
    \hline
    {\bf 1}&$\begin{array}{l}  xy + h_+y + g(x,z,u)=0, \\
h_+ \in \langle y,z,u \rangle, m:=v_{\bf w}(g). 
     \end{array}$
     & $\begin{array}{l}I_C=\langle y,z,u \rangle \\  {\bf w}=(0,m,1,1)\end{array}$  & $\frac{1}{m}(-1,1,1)$ &{\bf I} \\

    \hline
   {\bf 2}&$\begin{array}{l}  x^{k }y + zy+h_+y + g(x,z,u)=0, \\
h_+ \in \langle y,z,u \rangle, m:=v_{\bf w}(g), \\
x \nmid g_{v=m}, \textup{ and } k \ge 2.
\end{array}$
     & $\begin{array}{l}I_C=\langle y,z,u \rangle \\ {\bf w}=(0,m,1,1)\end{array}$  & $cA/m$ & {\bf II-1} \\
      \hline

{\bf 3}&$\begin{array}{l}  \left\{ \begin{array}{l}  x^{k'}t + zy+h_+y +  g_{v>m}=0, \\
t- (x^{k-k'}y + g')=0, \end{array} \right. \\
h_+ \in \langle y,z,u \rangle, m:=v_{\bf w}(g)
\textup{ with } g_{v=m}=x^{k'}  g', \\ k \ge 2, \textup { and } k \ge k' \ge 1. \end{array}$
     & $\begin{array}{l}I_C=\langle y,z,u,t \rangle \\ {\bf w}=(0,m,1,1,m+1)\end{array}$  & $\begin{array}{l} cA/m,\\ aw=k-k' \ge 0 \\ cA/(m+1), \\ aw=k' \ge 1.\end{array}$ &{\bf II-2} \\
     \hline

{\bf 4}&$\begin{array}{l}  \left\{ \begin{array}{l}  xt +h_+y +  g_{v>1}=0, \\
t- (x^{k-1}y + z)=0, \end{array} \right. \\
h_+\in \langle y, z,u\rangle, m':=v_{\bf w}(h_+y+g_{v>1}), \textup{ and } k \ge 2.
 \end{array}$
     & $\begin{array}{l}I_C=\langle y,z,u,t \rangle \\ {\bf w}=(0,m'-1,1,1,m')\end{array}$   & $\frac{1}{m'}(1,1,m'-1)$ &{\bf III}\\
     \hline
     
   {\bf 5} &$\begin{array}{l}  x^{k }y +h_+y + zu +  g_{>}=0 \\
h_+ \in \langle y,z,u \rangle, \textup{ and } k \ge 2.\end{array}$
     & $\begin{array}{l}I_C=\langle y,z,u \rangle \\ {\bf w}=(0,2,1,1)\end{array}$  & $cA/2$ &{\bf IV}\\
     \hline
     
{\bf 6}&$\begin{array}{l}  x^{k }y +y^2+ h_+y + z^2 + \epsilon u^2 + g_{>}=0 \\
h_+ \in \langle y,z,u \rangle^2, g_+  \in \langle z,u \rangle^3,\\
\epsilon=0 \textup{ or }1, \textup{ and } k \ge 2.\end{array}$
     & $\begin{array}{l}I_C=\langle y,z,u \rangle \\ {\bf w}=(0,2,1,1)\end{array}$  & $cA/2$& {\bf V} \\    
     \hline 
     
{\bf 7}&$\begin{array}{l}  x^{k }y +h_+y + z^2 +u^2 +  g_{+}=0 \\
h_+ \in \langle y,z,u \rangle,\ g_+ \in \langle z,u \rangle^3, k \ge 2. \end{array}$
     & $\begin{array}{l}I_C=\langle y,z,u \rangle \\ {\bf w}=(0,2,1,1)\end{array}$  & $cA/2$&{\bf VIII} \\
    \hline
    \end{tabular}
    \caption{classifications in the case Sing$(X)=cA$.} 
    \label{Table cA}
\end{table}
}

{\Small

\begin{longtable}[h]{|l|l|l|l|l|} 
    \hline
     {\bf No.}&$X\hookrightarrow  W$ and description & $\begin{array}{l} I_{C/W} \textup{ and }\\\textup{weights }{\bf w}\end{array}$ & $\begin{array}{l}  \textup{Non-Gor. }\\\textup{sing. of }Y\end{array}$ &  {\bf Ref.}\\
    \hline
   {\bf 8}&$\begin{array}{l}  y^2+x^2y+ h_+y+ z^2u+g_+=0, \\
h_+ \in \langle y,z,u \rangle^2, 
\textup{ and }g_+ \in \langle z, u \rangle^3. \end{array}$
     & $\begin{array}{l}I_C=\langle y,z,u \rangle \\ {\bf w}=(0,3,1,1)\end{array}$  & $cD/3$& {\bf VI-1} \\
     \hline

{\bf 9}&$\begin{array}{l}  \left\{ \begin{array}{l}  y^2+x^2t +h_+y + z^2u +  g_{+}=0, \\
t- (y + g')=0, \end{array} \right. \\
h_+ \in \langle y,z,u \rangle^2, g_+ \in \langle z, u \rangle^3, \textup{ and } g_{v=2}=x^2 g'. \end{array}$
     & $\begin{array}{l}I_C=\langle y,z,u,t \rangle \\ {\bf w}=(0,2,1,1,3)\end{array}$   & $cD/3$ &{\bf VI-2} \\
     \hline
     
{\bf 10}&$\begin{array}{l}  \left\{ \begin{array}{l}  y^2+x^{2}t +h_+y + z^2u +  g_{+}=0, \\
t- (xy + g')=0, \end{array} \right. \\
h_+ \in \langle y,z,u \rangle^2, g_+ \in \langle z, u \rangle^3, \textup{ and } g_{v=2}=x^2 g'. \end{array}$
     & $\begin{array}{l}I_C=\langle y,z,u,t \rangle \\ {\bf w}=(0,2,1,1,3)\end{array}$   & $\begin{array}{l}cD/3,\\ \frac{1}{2}(1,1,1)\end{array}$ &{\bf VI-2} \\
     \hline

{\bf 11}&$\begin{array}{l}  \left\{ \begin{array}{l}  y^2+x^{2}t +h_+y + z^2u +  g_{+}=0, \\
t- (x^{k-2}y + g')=0. \end{array} \right. \\
h_+ \in \langle y,z,u \rangle^2, 
g_+ \in \langle z, u \rangle^3,\\
k \ge 4, \textup{ and } g_{v=2}=x^2g' \textup{ with rank}(g'_2)=2. \end{array}$
     & $\begin{array}{l}I_C=\langle y,z,u,t \rangle \\ {\bf w}=(0,2,1,1,3)\end{array}$   & $\begin{array}{l}cD/3,\\ cA/2 \end{array}$&{\bf VI-2} \\
     \hline
     
{\bf 12}&$\begin{array}{l}  \left\{ \begin{array}{l}  y^2+x^{2}t +h_+y + z^2u +  g_{+}=0, \\
t- (x^{2}y + g')=0, \end{array} \right. \\
h_+ \in \langle y,z,u \rangle^2, 
g_+ \in \langle z, u \rangle^3,\\ \textup{ and } g_{v=2}=x^2g' \textup{ with rank}(g'_2)=1. \\
\end{array}$
     & $\begin{array}{l}I_C=\langle y,z,u,t \rangle \\ {\bf w}=(0,2,1,1,3)\end{array}$   & $\begin{array}{l}cD/3,\\ cAx/2 \end{array}$ &{\bf VI-2} \\
     \hline
     
{\bf 13}&$\begin{array}{l}  \left\{ \begin{array}{l}  y^2+x^{2}t +h_+y + z^2u +  g_{+}=0, \\
t- (x^{k-2}y + g')=0. \end{array} \right. \\
h_+ \in \langle y,z,u \rangle^2, 
g_+ \in \langle z, u \rangle^3,  k \ge 4,\\ g_{v=2}=x^2g' \textup{ with either } z^2 + xu^2\in g' \textup{ or } u^2+xz^2 \in g'.\end{array}$
     & $\begin{array}{l}I_C=\langle y,z,u,t \rangle \\ {\bf w}=(0,2,1,1,3)\end{array}$   & $\begin{array}{l}cD/3,\\ cD/2 \end{array}$ &{\bf VI-2} \\
     \hline

{\bf 14}&$\begin{array}{l}  \left\{ \begin{array}{l}  y^2+x^{2}t +h_+y + z^2u +  g_{+}=0, \\
t- (x^{k-2}y + g')=0. \end{array} \right. \\
h_+ \in \langle y,z,u \rangle^2, 
g_+ \in \langle z, u \rangle^3,  k = 4,\\ g_{v=2}=x^2g' \textup{ with rank } (g'_2)=1, \textup{ and } g'_3=0.\end{array}$
     & $\begin{array}{l}I_C=\langle y,z,u,t \rangle \\ {\bf w}=(0,2,1,1,3)\end{array}$   & $\begin{array}{l}cD/3,\\ cE/2 \end{array}$ &{\bf VI-2} \\
     \hline

{\bf 15}&$\begin{array}{l}  \left\{ \begin{array}{l}  y^2+xt +h_+y  +  g_{+}=0, \\
t- (xy + g')=0, \end{array} \right. \\
h_+ \in \langle y,z,u \rangle^2, g_+ \in \langle z, u \rangle^3, \textup{ and } g_{v=2}=x g' \textup{ with } g'_2 \ne 0. \end{array}$
     & $\begin{array}{l}I_C=\langle y,z,u,t \rangle \\ {\bf w}=(0,2,1,1,3)\end{array}$   & $\begin{array}{l} \frac{1}{3}(2,1,1),\\ \frac{1}{2}(1,1,1)\end{array}$ &{\bf VI-3} \\
     \hline

{\bf 16}&$\begin{array}{l}  \left\{ \begin{array}{l}  y^2+xt +h_+y  +  g_{+}=0, \\
t- (x^{k-1}y + g')=0. \end{array} \right. \\
h_+ \in \langle y,z,u \rangle^2, 
g_+ \in \langle z, u \rangle^3,\\
k \ge 3, \textup{ and } g_{v=2}=xg' \textup{ with rank}(g'_2)=2. \end{array}$
     & $\begin{array}{l}I_C=\langle y,z,u,t \rangle \\ {\bf w}=(0,2,1,1,3)\end{array}$   & $\begin{array}{l}\frac{1}{3}(2,1,1),\\ cA/2 \end{array}$&{\bf VI-3} \\
     \hline
     
{\bf 17}&$\begin{array}{l}  \left\{ \begin{array}{l}  y^2+xt +h_+y  +  g_{+}=0, \\
t- (x^{2}y + g')=0, \end{array} \right. \\
h_+ \in \langle y,z,u \rangle^2, 
g_+ \in \langle z, u \rangle^3,\\ \textup{ and } g_{v=2}=xg' \textup{ with rank}(g'_2)=1. \\
\end{array}$
     & $\begin{array}{l}I_C=\langle y,z,u,t \rangle \\ {\bf w}=(0,2,1,1,3)\end{array}$   & $\begin{array}{l}\frac{1}{3}(2,1,1),\\ cAx/2 \end{array}$ &{\bf VI-3} \\
     \hline
     
{\bf 18}&$\begin{array}{l}  \left\{ \begin{array}{l}  y^2+xt +h_+y  +  g_{+}=0, \\
t- (x^{k-2}y + g')=0. \end{array} \right. \\
h_+ \in \langle y,z,u \rangle^2, 
g_+ \in \langle z, u \rangle^3,  k \ge 4,\\ g_{v=2}=xg' \textup{ with rank } (g'_2)=1.\end{array}$
     & $\begin{array}{l}I_C=\langle y,z,u,t \rangle \\ {\bf w}=(0,2,1,1,3)\end{array}$   & $\begin{array}{l}\frac{1}{3}(2,1,1),\\ cD/2 \end{array}$ &{\bf VI-3} \\
     \hline
 
    {\bf 19}&$\begin{array}{l} \left\{ \begin{array}{l}  y^2+xt +h_+y +g_{+}=0, \\
t- (x^{k-1}y + g')=0, \end{array} \right. \\
h_+ \in \langle y,z,u \rangle^2, g_+ \in \langle z,u \rangle^4,  g_{v=2}=xg', \textup{ and } k \ge 2. \end{array}$
     & $\begin{array}{l}I_C=\langle y,z,u,t \rangle \\ {\bf w}=(0,2,1,1,4)\end{array}$   & $cAx/4$ & {\bf VI-4}\\
     \hline





     
{\bf 20}&$\begin{array}{l} \left\{ \begin{array}{l}  y^2+x^{2}t +h_+y +  g_{+}=0, \\
t- (x^{k-2}y + g')=0, \end{array} \right. \\
h_+ \in \langle y,z,u \rangle^2, g_+ \in x\langle z,u \rangle^2 + \langle z,u \rangle^3, \\
 g_{v=1}=x^2g',\textup{ and } k \ge 2.\end{array}$
     & $\begin{array}{l}I_C=\langle y,z,u,t \rangle \\ {\bf w}=(0,1,1,1,2)\end{array}$   & $cAx/2$ &{\bf VI-5}\\
     \hline

     
{\bf 21}&$\begin{array}{l} \left\{ \begin{array}{l}  y^2+x^{k'}t +h_+y +  g_{+}=0, \\
t- (x^{k-k'}y + g')=0, \end{array} \right. \\
h_+ \in \langle y,z,u \rangle^2, g_+ \in x\langle z,u \rangle^2 + \langle z,u \rangle^3, g_{v=1}=x^{k'}g',  \\
\textup{ some of } \{xz^2, xzu, xu^2\} \in g_+, \textup{ and } k\geq k'\geq 3. \end{array}$
     & $\begin{array}{l}I_C=\langle y,z,u,t \rangle \\ {\bf w}=(0,1,1,1,2)\end{array}$   & $cD/2$ &{\bf VI-5}\\
     \hline

{\bf 22}&$\begin{array}{l} \left\{ \begin{array}{l}  y^2+x^{3}t +h_+y +  g_{+}=0, \\
t- (x^{k-3}y + g')=0, \end{array} \right. \\
h_+ \in \langle y,z,u \rangle^2, g_+ \in x^2\langle z,u \rangle^2 + \langle z,u \rangle^3, g_{v=1}=x^3g', \\
\textup{ some of } \{z^3, z^2u, zu^2, u^3\} \in g_+
\end{array}$
     & $\begin{array}{l}I_C=\langle y,z,u,t \rangle \\ {\bf w}=(0,1,1,1,2)\end{array}$   & $cE/2$ &{\bf VI-5}\\
     \hline

{\bf 23}&$\begin{array}{l}   x^{2 }y +h_+y + z^2  +  g_{+}=0, \\
h_+ \in x  \langle y,z,u \rangle+ \langle y,z,u \rangle^2,\ g_+ \in x\langle z,u \rangle^2+\langle z,u \rangle^3. \end{array}$
     & $\begin{array}{l}I_C=\langle y,z,u \rangle \\ {\bf w}=(0,2,1,1)\end{array}$  & $cAx/2$ & {\bf IX-1} \\   
      \hline
      
{\bf 24} &$\begin{array}{l}  x^{k }y +h_+y + z^2  + xu^2+\epsilon xuz+ g_{+}=0 \\
h_+ \in x  \langle y,z,u \rangle+ \langle y,z,u \rangle^2,\ g_+ \in x^2\langle z,u \rangle^2+\langle z,u \rangle^3, \\
k \ge 3, \epsilon=0 \textup { or }1. \end{array}$
     & $\begin{array}{l}I_C=\langle y,z,u \rangle \\ {\bf w}=(0,2,1,1)\end{array}$  & $cD/2$ & {\bf IX-2} \\  
     \hline
    
{\bf 25}&$\begin{array}{l} x^{k }y +\lambda xy + h_+y + z^2  + g_{+}=0 \\
h_+ \in x^2 \langle y,z,u \rangle+ \langle y,z,u \rangle^2,  g_+ \in x^2\langle z,u \rangle^2+\langle z,u \rangle^3\\ 
k \ge 3, \lambda \textup{ is linear in  } y,u. \end{array}$
     & $\begin{array}{l}I_C=\langle y,z,u \rangle \\ {\bf w}=(0,2,1,1)\end{array}$ & $cD/2$ & {\bf IX-3}\\ 
     \hline

{\bf 26}&$\begin{array}{l} x^{3 }y +h_+y + z^2  +u^3+zu^2+ g_{+}=0 \\
h_+ \in x^2\langle y,z,u \rangle+ \langle y,z,u \rangle^2, \\  g_+ \in x^2\langle z,u \rangle^2+ x\langle z,u \rangle^3 +\langle z,u \rangle^4,  
k \ge 3. \end{array}$
     & $\begin{array}{l}I_C=\langle y,z,u \rangle \\ {\bf w}=(0,2,1,1)\end{array}$ & $cE/2$ & {\bf IX-4}\\ 
     \hline
     
{\bf 27}&$\begin{array}{l}  \left\{ \begin{array}{l}  x^{2}t +h_+y + z^2 +  g_{+}=0, \\
t- (x^{k-2}y + g')=0, \end{array} \right. \\
h_+ \in x\langle y,z,u \rangle+\langle y,z,u \rangle^2, g_+ \in x\langle z,u \rangle^2+\langle z,u \rangle^3, \\
 k\geq 2, \textup{ with }g_{v=1}=x^2g', \textup{ and } f_3 \textup{ is not a cube}.  \end{array}$
     & $\begin{array}{l}I_C=\langle y,z,u,t \rangle \\ {\bf w}=(0,1,1,1,2)\end{array}$   & $cAx/2$ &{\bf IX-5} \\

    \hline
{\bf 28}&$\begin{array}{l}  \left\{ \begin{array}{l}  x^{k'}t +h_+y + z^2 +  g_{+}=0, \\
t- (x^{k-k'}y + g')=0, \end{array} \right. \\
h_+ \in x\langle y,z,u \rangle+\langle y,z,u \rangle^2, g_+ \in x\langle z,u \rangle^2+\langle z,u \rangle^3, \\
 k\geq k'  \ge 3, \textup{ with }g_{v=1}=x^2g', \textup{some of }\{xy^2,xyu,xu^2 \} \in f.  \end{array}$
     & $\begin{array}{l}I_C=\langle y,z,u,t \rangle \\ {\bf w}=(0,1,1,1,2)\end{array}$   & $cD/2$ &{\bf IX-5} \\

    \hline
{\bf 29}&$\begin{array}{l}  \left\{ \begin{array}{l}  x^{3}t +h_+y + z^2 +  g_{+}=0, \\
t- (x^{k-3}y + g')=0, \end{array} \right. \\
h_+ \in x\langle y,z,u \rangle+\langle y,z,u \rangle^2, g_+ \in x\langle z,u \rangle^2+\langle z,u \rangle^3, \\
 k \ge 3, \textup{ with }g_{v=1}=x^2g', \textup{none of }\{xy^2,xyu,xu^2 \} \in f, \\ \textup{some of }\{y^3,y^2u,yzu,yu^2, zu^2,u^3 \} \in f,  \textup{ and } f_3 \textup{ is not a cube}.  \end{array}$
     & $\begin{array}{l}I_C=\langle y,z,u,t \rangle \\ {\bf w}=(0,1,1,1,2)\end{array}$   & $cE/2$ &{\bf IX-5} \\
    
    \hline
    \caption{classifications in the case Sing$(X)=cD$}

    \label{Table cD'}
\end{longtable}
}

{\Small

\begin{longtable}[h]{|l|l|l|l|l|} 
    \hline
    {\bf No.}&$X\hookrightarrow  W$ and description & $\begin{array}{l} I_{C/W} \textup{ and }\\\textup{weights }{\bf w}\end{array}$ & $\begin{array}{l}  \textup{Non-Gor. }\\\textup{sing. of }Y\end{array}$ &  {\bf Ref.}\\
    \hline
   {\bf 30}&$\begin{array}{l}  y^2+x^2y+ h_+y+ z^3+g_{+}=0\\
h_+ \in \langle y,z,u \rangle^2, g_+ \in x\langle z,u \rangle^3+\langle z,u \rangle^4.  \end{array}$
     & $\begin{array}{l}I_C=\langle y,z,u \rangle \\ {\bf w}=(0,3,1,1)\end{array}$  & $cD/3$ & {\bf VII-1} \\
     \hline
     
    {\bf 31}&$\begin{array}{l} \left\{ \begin{array}{l}  y^2+x^{2}t +h_+y + z^3 +  g_{+}=0, \\
t- (y + g')=0, \end{array} \right. \\
h_+ \in \langle y,z,u \rangle^2, g_+ \in x\langle z,u \rangle^3+\langle z,u \rangle^4, 
 g_{v=2}=x^{2} g'. \end{array}$
     & $\begin{array}{l}I_C=\langle y,z,u,t \rangle \\ {\bf w}=(0,2,1,1,3)\end{array}$  & $cD/3$ &{\bf VII-2}\\
      \hline
      
{\bf 32}&$\begin{array}{l} \left\{ \begin{array}{l}  y^2+x^{2}t +h_+y + z^3 +  g_{+}=0, \\
t- (xy + g')=0, \end{array} \right. \\
h_+ \in \langle y,z,u \rangle^2, g_+ \in x\langle z,u \rangle^3+\langle z,u \rangle^4, 
 g_{v=2}=x^{2} g'. \end{array}$
     & $\begin{array}{l}I_C=\langle y,z,u,t \rangle \\ {\bf w}=(0,2,1,1,3)\end{array}$  & $\begin{array}{l} cD/3,\\ \frac{1}{2}(1,1,1)\end{array}$ &{\bf VII-2}\\
      \hline
      
{\bf 33}&$\begin{array}{l} \left\{ \begin{array}{l}  y^2+x^{2}t +h_+y + z^3 +  g_{+}=0, \\
t- (x^{k-2}y + g')=0, \end{array} \right. \\
h_+ \in \langle y,z,u \rangle^2, g_+ \in x\langle z,u \rangle^3+\langle z,u \rangle^4,\\
g_{v=2}=x^{2} g',  \textup{ with  rank}(g_2')=2, \textup{ and } k \ge 4. \end{array}$
     & $\begin{array}{l}I_C=\langle y,z,u,t \rangle \\ {\bf w}=(0,2,1,1,3)\end{array}$  & $\begin{array}{l} cD/3,\\ cA/2\end{array}$  &{\bf VII-2}\\
     \hline
     
{\bf 34}&$\begin{array}{l} \left\{ \begin{array}{l}  y^2+x^{2}t +h_+y + z^3 +  g_{+}=0, \\
t- (x^{2}y + g')=0, \end{array} \right. \\
h_+ \in \langle y,z,u \rangle^2, g_+ \in x\langle z,u \rangle^3+\langle z,u \rangle^4, \\
g_{v=2}=x^{2} g' \textup{ with rank}(g_2')=1. \end{array}$
     & $\begin{array}{l}I_C=\langle y,z,u,t \rangle \\ {\bf w}=(0,2,1,1,3)\end{array}$  & $\begin{array}{l} cD/3,\\ cAx/2\end{array}$ &{\bf VII-2}\\
     \hline

{\bf 35}&$\begin{array}{l} \left\{ \begin{array}{l}  y^2+x^{2}t +h_+y + z^3 +  g_{+}=0, \\
t- (x^{k-2}y + g')=0, \end{array} \right. \\
h_+ \in \langle y,z,u \rangle^2, g_+ \in x\langle z,u \rangle^3+\langle z,u \rangle^4, \\
 g_{v=2}=x^{2} g', \textup{with }g' \ni z^2+xu^2, \textup{ or } u^2+xz^2.  \end{array}$
     & $\begin{array}{l}I_C=\langle y,z,u,t \rangle \\ {\bf w}=(0,2,1,1,3)\end{array}$  & $\begin{array}{l} cD/3,\\ cD/2\end{array}$ &{\bf VII-2}\\
     \hline

{\bf 36}&$\begin{array}{l} \left\{ \begin{array}{l}  y^2+x^{2}t +h_+y + z^3 +  g_{+}=0, \\
t- (x^{3}y + g')=0, \end{array} \right. \\
h_+ \in \langle y,z,u \rangle^2, g_+ \in x\langle z,u \rangle^3+\langle z,u \rangle^4, \\
 g_{v=2}=x^{2} g', \textup{with rank} (g_2') =1 \textup{ and } g'_3=0.  \end{array}$
     & $\begin{array}{l}I_C=\langle y,z,u,t \rangle \\ {\bf w}=(0,2,1,1,3)\end{array}$  & $\begin{array}{l} cD/3,\\ cE/2\end{array}$ &{\bf VII-2}\\
     \hline
     
{\bf 37}&$\begin{array}{l} \left\{ \begin{array}{l}  y^2+x^{2}t +h_+y + z^3 +  g_{+}=0, \\
t- (y + g')=0, \end{array} \right. \\
h_+ \in \langle y,z,u \rangle^2, \ g_+ \in x^2\langle z, u \rangle^2 + x\langle z, u \rangle^3 +\langle z, u \rangle^4 \\
g_{v=1}=x^2g'. \end{array}$
     & $\begin{array}{l}I_C=\langle y,z,u,t \rangle \\ {\bf w}=(0,1,1,1,2)\end{array}$  & $cAx/2$ &{\bf VII-3}\\
     \hline

     {\bf 38}&$\begin{array}{l}   x^{3}y +h_+y + z^2  + q+  g_{+}=0, \\
 h_+ \in x^2 \langle y,z,u \rangle +x \langle y,z,u \rangle^2 + \langle y,z,u \rangle^3,  \\ g_+ \in x^2\langle z, u \rangle^2 + x\langle z, u \rangle^3 +\langle z, u \rangle^4,  q=y^3 \textup{ or } u^3. \end{array}$
     & $\begin{array}{l}I_C=\langle y,z,u \rangle \\ {\bf w}=(0,2,1,1)\end{array}$ & $cE/2$ &{\bf X-1, X-2 } \\  
     \hline

    {\bf 39}&$\begin{array}{l}  \left\{ \begin{array}{l}  x^{3}t +h_+y + z^2 + q + g_{+}=0, \\
t- (x^{k-3}y + g')=0. \end{array} \right. \\
   h_+ \in  x^2\langle y,z,u \rangle+x\langle y,z,u \rangle^2+\langle y,z,u \rangle^3, \\
   \ g_+ \in x^2\langle z, u \rangle^2 + x\langle z, u \rangle^3 +\langle z, u \rangle^4, 
q=y^3 \textup{ or } u^3,  k\geq 3.
 \end{array}$
     & $\begin{array}{l}I_C=\langle y,z,u,t \rangle \\ {\bf w}=(0,1,1,1,2)\end{array}$   & $cE/2$ &{\bf X-3, X-4} \\ 
    \hline

    \caption{classifications when Sing$(X)=cE$}
    \label{Table cE}

\end{longtable}
}

{\Small

\begin{longtable}[h]{|l|l|l|l|l|} 
    \hline
    {\bf No.}&$X\hookrightarrow  W$ and description & $\begin{array}{l} I_{C/W} \textup{ and }\\\textup{weights }{\bf w}\end{array}$ & $\begin{array}{l}  \textup{Non-Gor. }\\\textup{sing. of }Y\end{array}$ &  {\bf Ref.}\\
    \hline
    {\bf Q1}&$\begin{array}{l}  xy + h_+y + g_{m}(x,z,u) + g_{>}(x,z,u)=0 \\
\subset \bC^4/\frac{1}{r}(r-1, 1, \alpha, r),\\
h_+ \in \langle y,z,u\rangle,\ g_m \in \langle z,u\rangle^m, \\
\gcd(\alpha,r)=1 \textup{ and } \gcd(m\alpha-1,mr)=1 \end{array}$
     & $\begin{array}{l}I_C=\langle y,z,u \rangle \\  {\bf w}=(0,m,1,1)\end{array}$  & $\frac{1}{mr}(1,m\alpha-1,mr-1)$ &{\bf I} \\
     \hline
    {\bf Q2}&$\begin{array}{l}  \left\{ \begin{array}{l}  xt +h_+y +  g_{v>1}=0, \\
t- (x^{k-1}y + z)=0, \end{array} \right. \\
\subset \bC^5/\frac{1}{r}(r-1, \alpha, 1, r,1)  \textup{ or }\bC^5/\frac{1}{r}(r-1, r, 1, \alpha,1),  
\\
h_+\in \langle y, z,u\rangle, k \ge 2,\\
m':=wt(h_+y+g_{v>1}), \gcd(m'\alpha-1,m'r)=1.
 \end{array}$
     & $\begin{array}{l}I_C=\langle y,z,u,t \rangle \\ {\bf w}=(0,1,1,1,m')\end{array}$   & $\begin{array}{l} \frac{1}{m'r}(m'\alpha-1,1,-1)\\ \textup{or }\frac{1}{m'r}(-1,1,m'\alpha-1)\end{array}$ &{\bf III}\\
     \hline

    {\bf Q3}&$\begin{array}{l}   x^ky +y^2+ h_+y +z^2+\epsilon u^2 + g_{>2}=0   \\
\subset \bC^4/\frac{1}{2}(1, 1, 0, 1) \\  
h_+ \in \langle y,z,u \rangle^2, \epsilon =0 \textup{ or } 1, \textup{ and odd } k \ge 2.
 \end{array}$
     & $\begin{array}{l}I_C=\langle y,z,u \rangle \\ \bw=(0,2,1,1)\end{array}$   & $ cAx/4 \textup{ in } \frac{1}{4}(2,1,3,1)$ &{\bf V} \\ 
     \hline

{\bf Q4} & $\begin{array}{l}   x^{k }y +h_+y + z^2 +u^2 +  g_{+}=0   \\
\subset \bC^4/\frac{1}{2}(1, 1, 1, 0) \\  
h_+ \in x\langle y,z,u \rangle+\langle y,z,u \rangle^2, g_+ \in \langle z,u \rangle^3, \\
\textup{and odd } k \ge 2.
 \end{array}$
     & $\begin{array}{l}I_C=\langle y,z,u \rangle \\ \bw=(0,2,1,1)\end{array}$   & $ cAx/4 \textup{ in } \frac{1}{4}(2,1,1,3)$ &{\bf VIII} \\ 
     \hline

    \caption{classifications in the case Sing$(X)=cDV/r$.} 
    \label{Table cA/r}
\end{longtable}
}


\bibliographystyle{amsalpha}
\bibliography{JKref}

\providecommand{\bysame}{\leavevmode\hbox to3em{\hrulefill}\thinspace}
\providecommand{\MR}{\relax\ifhmode\unskip\space\fi MR }
\providecommand{\MRhref}[2]{%
  \href{http://www.ams.org/mathscinet-getitem?mr=#1}{#2}
}
\providecommand{\href}[2]{#2}
\begin{thebibliography}{Yam18}

\bibitem[Cut88]{Cut88}
Steven Cutkosky, \emph{Elementary contractions of {G}orenstein threefolds},
  Math. Ann. \textbf{280} (1988), no.~3, 521--525. \MR{936328}

\bibitem[DL15]{DL15}
Tobias Dorsch and Vladimir Lazi\'c, \emph{A note on the abundance conjecture},
  Algebr. Geom. \textbf{2} (2015), no.~4, 476--488. \MR{3403237}

\bibitem[Duc16]{Duc16}
Tom Ducat, \emph{Divisorial extractions from singular curves in smooth
  3-folds}, Internat. J. Math. \textbf{27} (2016), no.~1, 1650005, 23.
  \MR{3454683}

\bibitem[GR09]{GunningRossi}
Robert~C. Gunning and Hugo Rossi, \emph{Analytic functions of several complex
  variables}, AMS Chelsea Publishing, Providence, RI, 2009, Reprint of the 1965
  original. \MR{2568219}

\bibitem[Hay99]{Hayakawa99}
Takayuki Hayakawa, \emph{Blowing ups of {$3$}-dimensional terminal
  singularities}, Publ. Res. Inst. Math. Sci. \textbf{35} (1999), no.~3,
  515--570. \MR{1710753}

\bibitem[Hay00]{Hayakawa00}
\bysame, \emph{Blowing ups of 3-dimensional terminal singularities. {II}},
  Publ. Res. Inst. Math. Sci. \textbf{36} (2000), no.~3, 423--456. \MR{1781436}

\bibitem[Kaw96]{Kaw96}
Yujiro Kawamata, \emph{Divisorial contractions to {$3$}-dimensional terminal
  quotient singularities}, Higher-dimensional complex varieties ({T}rento,
  1994), de Gruyter, Berlin, 1996, pp.~241--246. \MR{1463182}

\bibitem[Kaw02]{Kawakita02}
Masayuki Kawakita, \emph{Divisorial contractions in dimension three which
  contract divisors to compound {$A_1$} points}, Compositio Math. \textbf{133}
  (2002), no.~1, 95--116. \MR{1918291}

\bibitem[Kaw05]{Kawakita05}
\bysame, \emph{Three-fold divisorial contractions to singularities of higher
  indices}, Duke Math. J. \textbf{130} (2005), no.~1, 57--126. \MR{2176548}

\bibitem[Kaw08]{Kawamata08}
Yujiro Kawamata, \emph{Flops connect minimal models}, Publ. Res. Inst. Math.
  Sci. \textbf{44} (2008), no.~2, 419--423. \MR{2426353}

\bibitem[Ke92]{flips}
J\'anos Koll\'ar~(ed), \emph{Flips and abundance for algebraic threefolds},
  Soci\'et\'e{} Math\'ematique de France, Paris, 1992, Papers from the Second
  Summer Seminar on Algebraic Geometry held at the University of Utah, Salt
  Lake City, Utah, August 1991, Ast\'erisque No. 211 (1992). \MR{1225842}

\bibitem[KM92]{km92}
J.~Koll\'{a}r and S.~Mori, \emph{Classification of three-dimensional flips}, J.
  Amer, Math. Soc. \textbf{5} (1992), 533--703.

\bibitem[KM98]{km}
J\'anos Koll\'ar and Shigefumi Mori, \emph{Birational geometry of algebraic
  varieties}, Cambridge Tracts in Mathematics, vol. 134, Cambridge University
  Press, Cambridge, 1998, With the collaboration of C. H. Clemens and A. Corti,
  Translated from the 1998 Japanese original. \MR{1658959}

\bibitem[Mor82]{Mori82}
Shigefumi Mori, \emph{Threefolds whose canonical bundles are not numerically
  effective}, Ann. of Math. (2) \textbf{116} (1982), no.~1, 133--176.
  \MR{662120}

\bibitem[Tzi03]{Tz03}
Nikolaos Tziolas, \emph{Terminal 3-fold divisorial contractions of a surface to
  a curve. {I}}, Compositio Math. \textbf{139} (2003), no.~3, 239--261.
  \MR{2041612}

\bibitem[Tzi05a]{Tz05b}
\bysame, \emph{Families of {$D$}-minimal models and applications to 3-fold
  divisorial contractions}, Proc. London Math. Soc. (3) \textbf{90} (2005),
  no.~2, 345--370. \MR{2142131}

\bibitem[Tzi05b]{Tz05a}
\bysame, \emph{Three dimensional divisorial extremal neighborhoods}, Math. Ann.
  \textbf{333} (2005), no.~2, 315--354. \MR{2195118}

\bibitem[Tzi10]{Tz10}
\bysame, \emph{Three-fold divisorial extremal neighborhoods over {$cE_7$} and
  {$cE_6$} compound {D}u{V}al singularities}, Internat. J. Math. \textbf{21}
  (2010), no.~1, 1--23. \MR{2642984}

\bibitem[Yam18]{Yamamoto}
Yuki Yamamoto, \emph{Divisorial contractions to {$cDV$} points with discrepancy
  greater than 1}, Kyoto J. Math. \textbf{58} (2018), no.~3, 529--567.
  \MR{3843389}

\end{thebibliography}

\end{document}